\documentclass[a4paper]{article}

\usepackage[pages=all, color=black, position={current page.south}, placement=bottom, scale=1, opacity=1, vshift=5mm]{}

\usepackage[margin=1in]{geometry} 

\usepackage{amsmath}
\usepackage{amsthm}
\usepackage{amssymb}
\usepackage{tcolorbox}
\usepackage{commath}
\usepackage{amsfonts}
\usepackage{mathtools}
\usepackage{mathabx}
\usepackage{tabulary}
\usepackage{booktabs}
\usepackage[toc,page]{appendix}
\usepackage[mathscr]{euscript}
\usepackage{hyperref}
\usepackage{mathrsfs}
\usepackage[utf8]{inputenc}
\usepackage{hyperref}
\hypersetup{
	unicode,
	pdfauthor={Author One, Author Two, Author Three},
	pdftitle={A simple article template},
	pdfsubject={A simple article template},
	pdfkeywords={article, template, simple},
	pdfproducer={LaTeX},
	pdfcreator={pdflatex}
}

\usepackage[sort&compress,numbers,square]{natbib}

\usepackage{bm}

  \newcommand\vertarrowbox[3][6ex]{%
  \begin{array}[t]{@{}c@{}} #2 \\
  \left\uparrow\vcenter{\hrule height #1}\right.\kern-\nulldelimiterspace\\
  \makebox[0pt]{\scriptsize#3}
  \end{array}%
}
  \usepackage[shortlabels]{enumitem}
\let\olditem\item
\newlist{methods}{itemize}{1}
\setlist[methods]{%
    align=right,
    before=\changeitem,
    font=\bfseries,
    after=\let\item\olditem
}
\newcommand*{\changeitem}{%
    \renewcommand*{\item}[1][]{%
        \olditem[##1 :]
    }%
}
\theoremstyle{plain}
\newtheorem{theorem}{Theorem}
\newtheorem{corollary}[theorem]{Corollary}

\newtheorem{proposition}[theorem]{Proposition}

\theoremstyle{definition}
\newtheorem{definition}[theorem]{Definition}
\newtheorem{example}[theorem]{Example}
\newtheorem{remark}[theorem]{Remark}

\usepackage[mathscr]{euscript}
\usepackage{graphicx, color}
\graphicspath{{fig/}}
\newcommand{\ran}{\rangle}
\newcommand{\lan}{\langle}
\usepackage{algorithm, algpseudocode} 
\usepackage{mathrsfs} 

\newcommand\restr[2]{\ensuremath{\left.#1\right|_{#2}}}

\newcommand{\z}{\zeta}

\newcommand{\mv}{\mathcal{V}}

\newcommand{\til}{\Tilde}

\newcommand{\dl}{\partial}

\usepackage{bbm}
  \usepackage{nicematrix}
  \usepackage{pdfpages}

\usepackage{lscape}
\usepackage{subfig}
\usepackage{float}
\usepackage{plain}
\usepackage{lipsum}
\newcommand{\rmspan}{\mathrm{span}}
\newcommand{\rank}{\mathrm{rank}}

\newcommand{\defeq}{\overset{\mathrm{def}}{=}}

\newcommand{\trans}{{}^t\!}

\newcommand{\Dcal}{ {\mathcal D}}

\newcommand{\Vcal}{ {\mathcal V}}

\newcommand{\Acal}{ {\mathcal A}}

\newcommand{\Ccal}{ {\mathcal C}}

\newcommand{\Pcal}{ {\mathcal P}}
\newcommand{\Qcal}{ {\mathcal Q}}
\newcommand{\Zcal}{ {\mathcal Z}}

\newcommand{\Gbb}{ {\mathbb G}}
\usepackage{verbatim}

\newcommand{\Card}{{\rm Card}}

\newcommand{\R}{{\mathbb R}}

\newcommand{\N}{{\mathbb N}}
\newcommand{\und}{\frac{1}{2}}

\newcommand{\tendsto}[2]{\xrightarrow[{#1} \rightarrow {#2}]{}}%
\renewcommand{\leq}{\leqslant}
\renewcommand{\geq}{\geqslant}

\newcommand{\bfr}{\bm{r}}
\newcommand{\bfu}{\bm{u}}
\newcommand{\bfQ}{\bm{Q}}

\newcommand{\dvolsr}{\mathrm{dvol}_{\mathrm{SR}}}
\newcommand{\dvol}{\mathrm{dvol}}

\title{Approximating sub-Riemannian structures by Riemannian metrics and spectral convergence}
\author{L\'eo HARAKEH\footnote{Laboratoire de Math\' ematiques Jean Leray, Universit\'e de Nantes-CNRS, (\texttt{Leo.Harakeh@univ-nantes.fr}),
    also known under the name Mohammad Harakeh.} \and Luc HILLAIRET\footnote{Institut Denis Poisson, Universit\'e d'Orl\'eans-Université de Tours-CNRS, France (\texttt{luc.hillairet@univ-orleans.fr}).}}

\begin{document}
	\maketitle
	
	\begin{abstract}
          Sub-Riemannian structures can be seen as singular limits of Riemannian structures. Starting from a sub-Riemannian structure, we introduce
          an approximation scheme that is naturally associated with it. We prove that it induces a volume form in the limit that we compare to Popp's
          volume. We then prove that the spectrum of the family of Riemannian Laplacians associated to the approximation scheme converges to the
          spectrum of the sublaplacian.

	\end{abstract}

	\tableofcontents
	\section{Introduction}
        A sub-Riemannian structure on a manifold $M$ can be locally defined by choosing vector fields $(X_i)_{i=1, \cdots d}$ and a
        metric defined on the subbundle $V$
        of the tangent bundle generated by the $X_i$'s. When the vector fields satisfy the so-called \textit{Hörmander's bracket generating condition}
        there is a distance associated with this structure (see section \ref{framework} and
        \cite{agrachev2019comprehensive, jean2012control, montgomery2002tour} for background on sub-Riemannian geometry). It is known that this
        distance can be constructed by considering a singular limit of Riemannian metrics. Roughly speaking, its suffices to
        construct a sequence of metrics along which the length of directions orthogonal to $V$ tends to $+\infty$. This process allows to address different
        sub-Riemannian quantities by looking at how the latter are approximated by the sequence of the corresponding Riemannian ones. In this paper,
        we will be particularly interested in spectral quantities and we will thus consider the sequence of associated Riemannian Laplace operators and
        study how it behaves under such a degeneration. Such an approach is not new: for instance, it has been used in \cite{Melrose_wave} to study
        the heat and wave kernels, and in \cite{Rizzi_Rossi2021} for the heat content, see also \cite{ge1993collapsing, rumin2000sub, AlbinQuan2022}.
        It should be noted that the two former references consider very general sub-Riemannian structures with a very general approximation scheme,
        whereas the latter ones restrict to contact setting and use an approximation scheme that retains some feature from the geometry.
        This is the approach we want to pursue here. We will construct an approximating sequence of Riemannian metrics that is, in some sense,
        tailored to the sub-Riemannian setting. In particular we will build the approximation scheme keeping in mind
        the dilations associated to the structure (see \cite{bellaiche1996sub, jean2012control}). Indeed, these dilations, and the concept of
        privileged coordinates that is closely related, are deeply related to the local structure and can be used e.g. to obtain asymptotics
        for the heat kernel (see \cite{colindeverdiere:hal-02535865}).
        This will lead us to the approximation scheme and the sequence of Riemannian metrics $(g^h)_{h>0}$ defined in section \ref{subsec:approx}.
        Several questions are then naturally
        associated with this sequence. The first one is to understand how the Riemannian volume degenerates when $h$ goes to $0$.
        Indeed, contrary to the Riemannian case,
        there is no such thing as a canonical sub-Riemannian volume
        \cite{montgomery2002tour, agrachev2012hausdorff,gromov1996carnot,mitchell1985carnot, GhezziJean2015}). Several natural choices can be made
        e.g. Hausdorff, spherical Hausdorff, or Popp's volume and it was interesting to determine whether our approximation picked up a specific volume
        in the limit.

        Our first result is the following theorem in which we use $d$ for the dimension of the manifold and $\bfQ$ for the Hausdorff dimension of
        the sub-Riemannian structure in the equiregular region (see section \ref{frmwrk} for a precise definition).

        \begin{theorem}\label{thm:intro_vol}
            Let $(g^h)_{h>0}$ be the family of Riemannian metrics defined in section \ref{subsec:approx}, then
            $h^{d-\bfQ} \dvol{g^h}$ converges to Popp's volume uniformly on any compact set of the equiregular region.
          \end{theorem}

          Popp's volume has been studied in \cite{barilari2013formula} in which a notion of adapted frame was introduced and formulas were given for Popp's
          volume in such an adapted frame. We observe that our approach gives an effective way of computing Popp's volume in any coordinate frame.

          Our second result will be the spectral convergence of the sequence of Laplace operators associated with the Riemannian metric $g_h$. Observe that,
          in order to define a sub-Riemannian Laplace operator, we have to choose a smooth volume on $M$. 

          \begin{theorem}\label{thm:intro_spectral_convergence}
            Assume $M$ is compact and the sub-Riemannian structure is equiregular then, for any $k$, the $k$-th eigenvalue of the
            Riemannian Laplace operator associated with $g_h$ converges, when $h$ tends to $0$, to the $k$-th eigenvalue of the
            sub-Riemannian Laplace operator associated with Popp's volume.
          \end{theorem}

          The assumption of equiregularity is needed to ensure that Popp's volume is smooth on $M$. Indeed when it is not satisfied, the volume
          that is associated to our approximating sequence blows up near the singular region. It follows that before
          studying its spectral properties, already studying
          the self-adjointness of the sub-Riemannian Laplace operator needs to be addressed (see \cite{boscain2013laplace,prandi2018quantum,
            franceschi2020essential} for related questions). Deciding whether our approximation scheme picks up a particular extension
          in the limit and studying the spectral convergence in this setting would be an interesting question.
          
          The paper is organised as follows. After introducing the framework in section \ref{framework}, we introduce the approximation scheme.
          In section \ref{mstrbombastiv}, we study the volume form $\text{dvol}g^h$ and prove that it induces in the limit, a volume form,
          that is smooth on every connected component of the equiregular region and blows up on the singular one.
          We then compare this volume form to the Popp's volume in subsection \ref{rlnwzpopp} and prove theorem \ref{thm:intro_vol}.
          Finally, in section \ref{kitrez}, we study the convergence of spectrum of the family $\Delta_h$ and we prove a spectral convergence result
          in two settings. The first case is when the operators are defined with respect to a fixed volume form. No equiregularity assumption is needed
          but the coresponding operator at $h>0$ is not the Laplace operator associated with $g_h$. The second case is for the sequence of Laplace
          operators associated with $g_h$ in the equiregular region. The resulting theorem is \ref{thm:intro_spectral_convergence}.
          Both results rely on a uniform subelliptic estimate that we prove in the appendix, building upon a classical strategy (see
          \cite{kohn73, helffer20052}).   
          
          \subsection*{Acknowledgments :} We have benefitted from numerous discussions with several colleagues. We are grateful to
          Frédéric Jean and Emmanuel Trélat for sharing their insight on sub-Riemannian geometry. Both authors acknowledge the support of the
          ANR program ADYCT (grant ANR-20-CE40-0017).

     \section{General setting and notations}\label{framework}
      For any $n\in\mathbb N^*$, for any $\bfu=(u_1,...,u_n)\in\mathbb R^n$, we denote by $|\bfu|_n$ the Euclidean norm on $\R^n:~{|\bfu|_n^2=\sum\limits_{i=1}^nu_i^2}$. For any vector field $X,$ we will use both notations $X_m$ and $X(m).$ For a matrix $A, \trans A$ will denote the transposed matrix of $A.$
      \subsection{Framework}\label{frmwrk}
      We first recall some definitions from sub-Riemannian geometry. We refer to \cite{jean2012control,montgomery2002tour,agrachev2012hausdorff,
        agrachev2019comprehensive} for a more complete treatment.\\

      Let $M$ be a smooth orientable connected manifold of dimension $d$. We consider $N_0$ smooth vector fields $\Ccal_0=\{X^{01},...,X^{0N_0}\}$, and we define the distributions 
     \begin{align*}
       &~\Dcal_0 \,=\,\rmspan \left\{ X^{01},\cdots,\, X^{0N_0} \right\}\\
       \forall i\geq 0,& ~\Dcal_{i+1}\,=\,\Dcal_i\,+\,[\Dcal_0,\Dcal_i]. 
     \end{align*}
     H\"ormander's bracket condition at point $m$ stipulates that there exists a smallest integer $r(m)$ such that ${\Dcal_{r(m)} \,=\,T_m M}$. With this definition,
     we observe that the step of the sub-Riemannian structure at $m$ is $r(m)+1$. We assume that $r(m)$ admits a maximum on $M$ that we denote by $\bfr$;
     i.e
     \[
       \bfr=\mathrm{max}\{r(m); m\in M\}.
     \]
     This assumption will always be made througout the paper. When the manifold is compact, it is equivalent to ask that H\"ormander's condition
     holds at every point (by upper-semicontinuity of $\bfr(m)$).

     We denote by $(n_i(m))_{i=0,\cdots,\bfr}$ the so-called growth vector at $m$:
     \[
      \forall m\in M,~\forall i=0,\cdots \bfr,~~n_i(m) \defeq \dim \Dcal_i(m). 
     \]

     A point $m\in M$ is called equiregular if the growth vector is locally constant near $M$. Any other point is called singular.
     The manifold can be partitioned into the singular points $\mathcal{Z}$ and the equiregular region $M\setminus \mathcal{Z}$.
     The structure is called equiregular if $\mathcal{Z}$ is empty.
     
     \begin{remark}
       It is common, in sub-Riemannian geometry, to start with the distribution $\Dcal_0$ rather than with vector fields that generate it.
       We observe here that our construction will depend on the choice of these vector fields and not solely on the distribution.
       Using partitions of unity and the fact that, locally,
       we can always find $N_0$ vector fields that generates $\Dcal_0$, we can always construct the set $\Ccal_0$.
       Observe however, that, since the vector fields $(X^{0i})$ are globally defined, $N_0$ will, in general, be much greater that $n_0(m)$. 
     \end{remark}
     
     We denote by $g^0$ the sub-Riemannian metric associated with the set of vector fields $\Ccal_0$: for any $m\in M$ and $X_m\in T_mM$, we have 
     \begin{equation}\label{g^inf}
         \begin{split}
             g^{0}_m(X_m) \defeq \inf\left\{|\bfu|_{N_0}^2;\bfu\in\mathbb R^{N_0},\sum_{i=1}^{{N_0}}u_iX^{0i}_m=X_m\right\}.
         \end{split}
       \end{equation}
       We use the convention that the infimum over $\emptyset$ is $ +\infty$ so that $g^0_m(X_m)=+\infty$ whenever $X_m\notin \Dcal_0(m)$.

       For any absolutely continuous curve $\gamma:~[0,1]\rightarrow M$, we can define the length
       \[
         \ell_{SR}(\gamma) \defeq \int_0^1 \left(g^0_{\gamma(t)}(\dot{\gamma}(t))\right)^{\und}\, dt.
       \]

       Using Chow-Rashevski's theorem, it is proved that the following formula defines a distance on $M$:
       \[
         \forall x,y\in M,~~d^0(x,y)\defeq \inf \left \{ \ell_{SR}(\gamma),~~\gamma~\text{is absolutely continuous},~\gamma(0)=x,~\gamma(1)=y ~\right \},
       \]
       see \cite{agrachev2019comprehensive,jean2012control}.
       
       \subsection{Enumerating the brackets}
       First, it will be convenient to identify $\Ccal_0$ with the set $\{1,\cdots,N_0\}$
       so that we will slightly abuse notations and see $\Ccal_0$ both as a set of vector fields and as a set of indices.
     Let $k\in \N$ and $I$ a multiindex in the $k$-fold product $\Ccal_0^k$:
\[
  I\defeq (i_1,\cdots, i_k), 
\]
we denote by $X_I$ the corresponding bracket of length $k$:
\[
  X_I \,\defeq\, \left[X^{0i_1},\left[ X^{0i_2} ,\left[ \cdots [X^{0i_{k-1}},X^{0i_k}]\right]\cdots \right]\right].
\]
For any $k\geq 1$, we define
\[
  \Ccal_k \,\defeq\, \left\{ I\in \Ccal_0^{k+1},~~\exists m\in M,~(X_I)_m \neq 0 \right\}. 
\]
We now set $N_k\defeq \Card \,\Ccal_k$ and fix some enumeration of $\{1,\cdots N_k\}$ of $\Ccal_k$. Each index $j$ corresponds to a (unique) multiindex $I_j$ and
we write, for $i\geq 0$, $X^{ij}\,=\,X_{I_j}$, in which $I_j$ is of length $i+1$ (in $\Ccal_i)$.

\begin{remark}
  We insist that, except for $k=0$, it is important to see $\Ccal_k$ as a set of indices and not as a set of vector fields. Indeed, the correspondence between
  $(j,k)$ and $X^{kj}$ is, in general, not injective.
\end{remark}

Summarizing the notation: for any $i,j$, $X^{ij}$ is the $j$-th bracket of length $i$ in our chosen enumeration. 

By construction, we have
\[
  \forall m\in M,~\forall k\leq \bfr,~~\mathcal{D}_k(m) = \operatorname{span}\{\, X^{ij}_m \;:\; 0 \le i \le k,\ 1 \le j \le N_i \,\}.
\]
In particular, using the uniform H\"{o}rmander's condition and the definition of $M$, we have
\begin{equation}
        \label{hc} \forall m\in M,~~T_m M = \operatorname{span}\{\, X^{ij}_m \;:\; 0 \le i \le \mathbf r,\ 1 \le j \le N_i \,\}.
\end{equation}


   

\subsection{Approximation scheme}\label{subsec:approx}
In this section, we define a Riemannian approximation scheme that is tailored to the sub-Riemannian structure in the sense that is compatible with
 the natural associated dilations (see \cite{agrachev1989local,colindeverdiere:hal-02535865}). The construction is also reminiscent of the results by Nagel-Stein-Wainger (see \cite{nagel1985balls}) and of \cite{capogna2016regularity}.

Let $N=\sum\limits_{i=0}^{\bfr} N_i$. For $\bfu\in\mathbb R^N$, we write $\bfu=(\hat{u}_0,\hat{u}_1,...,\hat{u}_{\bfr}),$ where each $\hat{u}_i$ is of length $N_i$ so that

\[
  \forall i=0,\cdots \bfr,~~\hat{u}_i=(u_{i1},\cdots, u_{iN_i}).
\]
We will use both notations :

\[
  \bfu \,=\, (\hat{u}_0,\cdots,\hat{u}_{\bfr})\,=\,(u_{ij})_{\substack{0\leq i\leq \bfr \\ 1\leq j\leq N_i}}
\]

For $m\in M$, let $\sigma_m$ be the map $\sigma_m:\mathbb{R}^N\rightarrow T_mM,$ such that for $\bfu=(u_{ij})_{\substack{0\leq i\leq \bfr \\ 1\leq j\leq N_i}}$, 
     \begin{equation}
         \label{sigma}\sigma_m(\bfu)=\sum_{i=0}^{\bfr} \sum_{j=1}^{N_i}u_{ij}(m)X_m^{ij}.
       \end{equation}
       Using (\ref{hc}), for any $m$, $\sigma_m$ is surjective. In the language of bundles, we have defined a surjective bundle morphism from
       $M\times \R^N$ onto $TM$.
       
For all $h>0$, we define the (anisotropic) dilation $\delta_h$ as the endomorphism of $\R^N$ defined by:

$$
\forall \bfu\in \R^N,~~\delta_h(\bfu)=(\hat{u}_0,h^{-1}\hat{u}_1,h^{-2}\hat{u}_2,...,h^{-\bfr}\hat{u}_{\bfr}).
$$  
      
For $m\in M$, and $h>0$, we define the family of Riemannian metrics $g^h_m$ on $T_mM$ by
\begin{equation}\label{g^s}
\forall X_m\in T_mM,~ g_m^h(X_m)\,\defeq\,\inf \left\{ |\delta_h\bfu|_N^2;\, \bfu\in\mathbb{R}^N,\sigma_m(\bfu)=X_m \right\}. 
\end{equation}

\begin{remark}
  In order to define a Riemannian metric using the preceding formulas, we actually do not need to take into account all
  the vectors in the enumeration $(\mathcal{C}_k)$, it suffices to take a family that locally generates the tangent space.
  In \cite{AMY}, very general structures defined by vector fields are constructed. Our construction would associate a natural family of
  metrics. It would be interesting to see how they fit in their framework. 
\end{remark}

Let $q_h$ be the quadratic form that is defined on $\R^N$ by 
$q_h(\bfu)\,=\,|\delta_h\bfu|_N^2$. This quadratic form is definite positive
so it corresponds to a Euclidean scalar product. We denote by $\perp_h$
the corresponding orthogonal complement operation, and by $\tilde{\sigma}_m$
the restriction of $\sigma_m$ to $(\ker \sigma_m)^{\perp_h}$. By construction,
$\tilde{\sigma}_m$ is an isomorphism onto $T_mM$. A straightforward computation based
on Pythagoras' theorem shows that
\[
  g_m^h(X_m)\,=\,q_h\left((\tilde{\sigma}_m)^{-1}(X_m)\right).
\]
This shows that, for any fixed $m$, $g_m^h$ is a Euclidean quadratic form.
The smoothness in $m$ will follow from the computations in theorem \ref {Tgitoa}, so that $g^h$ is indeed a Riemannian metric.

\begin{proposition}
    Under the previous assumptions, the following assertions hold true.
   \begin{enumerate}
     \item The function $h\mapsto g^h$ is monotonically decreasing for any $h>0$.
      \item Fix $m\in M$. Then, $\restr{g^h_m}{\Dcal_0}\,\leq\,\restr{g^0_m}{\Dcal_0}$.
    \item  Fix $m\in M$. Then we have $\lim\limits_{h\rightarrow 0}\restr{g^h_m}{\Dcal_0}\,=\,\restr{g^0_m}{\Dcal_0}.$
 \end{enumerate}  
 \end{proposition}
 \begin{proof}
     \begin{enumerate}
     \item let $(m,X_m)\in TM$ and fix $\bfu=(\hat{u}_0,...,\hat{u}_{\bfr})\in\mathbb R^N$ such that $\sigma_m(\bfu)=X_m$.
\[
\forall 0<h_1<h_2,~g_m^{h_2}(X_m)\leq |\delta_{h_2} \bfu|_N^2=\sum_{i=0}^{\bfr} h_2^{-2i}|\hat{u}_i|^2<\sum_{i=0}^{\bfr} h_1^{-2i}|\hat{u}_i|^2=|\delta_{h_1} \bfu|_N^2.
\] 
Taking the infimum over all $\bfu\in\mathbb R^N$ satisfying $\sigma_m(\bfu)=X_m$, we get that, for all $0<h_1<h_2$,
\[g^{h_2}_m(X_m)\,\leq\,g^{h_1}_m(X_m).
\]
         \item  We can see this by observing that, for a fixed $m\in M$ and $X_m\in \Dcal_0$, we have 
\begin{equation*}
         g^0_m(X_m)=\inf\left\{|\delta_h \bfu|_N^2; \bfu=(\hat{u}_0,0,...,0)\in\mathbb R^N, \hat{u}_0\in\mathbb R^{N_0}, \sigma_m(\bfu)=X_m\right\},
 \end{equation*} 
 and that 
$$\left\{|\delta_h \bfu|_N^2; \bfu=(\hat{u}_0,0,...,0)\in\mathbb R^N, \hat{u}_0\in\mathbb R^{N_0}, \sigma_m(\bfu)=X_m\right\}\subset\left\{|\delta_h\bfu|_N^2; \bfu\in\mathbb R^N,\sigma_m(\bfu)=X_m\right\}.$$
 \item  Fix $m\in M$ and $X_m\in \Dcal_0$. For any $h>0$, there exists $\bfu(h)=(\hat{u}_0(h),...,\hat{u}_r(h))\in\mathbb R^N$ such that $\sigma_m(\bfu(h))=X_m$ and 
\begin{equation*}
  \label{e5rpropp} g^h_m(X_m)=|\bfu(h)|_N^2=|\hat{u}_0(h)|_{{N_0}}^2\,+\,
  \sum_{i=1}^ {\bfr} h^{-2i}|\hat{u}_i(h)|_{N_i}^2\leq g^0_m(X_m).
 \end{equation*} 
 So, we get that $\lim_{h\rightarrow0} \hat{u}_i(h)=0$ for any $i=1,...,{\bfr}$. Thus, we have 
 $$X_m=\sigma_m(\bfu(h))=\lim_{h\rightarrow0}\sigma_m(\bfu(h))=\sigma_m((\hat{u}_0(0),0,...,0)).$$
 Finally, we have that
 $$|\hat{u}_0(h)|_{N_0}^2\leq |\hat{u}_0(h)|_{{N_0}}^2\,+\,
\sum_{i=1}^{\bfr} h^{-2i}|\hat{u}_i(h)|_{N_i}^2=g^h_m(X_m)\leq g^0_m(X_m).
$$
Therefore, as $h\rightarrow0$, and by definition of $g^0$, we get that
 $$g^0_m(X_m)\leq \lim_{h\rightarrow0} g^h_m(X_m)\leq g^0_m(X_m).$$
 We conclude.
     \end{enumerate}
 \end{proof}

 The fact that the Riemannian distance associated with $g^h$ converges to the sub-Riemannian distance $d^0$ is
 a standard fact in sub-Riemannian geometry (see \cite{agrachev2001subanalyticity}\cite{ge1993collapsing}\cite{gromov1999metric}\cite{varopoulos1990small}). We thus obtain the following theorem. 

\begin{theorem} 
  Let $d^h$ be the distance associated to the Riemannian metric $g^h$ (following the definition of $d^0)$, and $d^0$ the sub-Riemannian distance defined in section \ref{frmwrk}.
  When $h$ goes to $0$, the family $(d^h)_{h>0}$ converges to $d^0$ uniformly on every compact set of $M\times M$.
\end{theorem}

\begin{remark}
  In \cite{nagel1985balls} various families of distances are defined that also implement the fact that vectors may scale differently with
  respect to $h$. These are closely related to our setting. One difference is that we use a Riemannian approach where the authors of
  \cite{nagel1985balls} adopt a more control-theoretic point of view.
\end{remark}

One major difference between sub-Riemannian and Riemannian geometry, is the lack, in the former setting of a natural canonical volume form.
Our main interest in the sequel will be to determine whether a particular volume is associated to the family $(g^h)_{h>0}$. 
It should be noted that this volume will a priori depend on the choice of the initial vector fields $\Ccal_0$. 
  
\section{The Induced Volume Form $\mathcal{P}_o$}\label{mstrbombastiv}
In this section, we prove that the previous approximation scheme induces a volume form on $M$ that we compare to Popp's volume. Thus, 
we consider the volume form $\text{dvol}g^h$ and study its behaviour as $h\rightarrow0$. In this section, everything will be purely local 
and we will work near a point $m$ and in a local frame $(Z_1,...,Z_d)$ of $TM$. 

In this frame, the volume form $\text{dvol}g^h$ is given by 
\begin{equation}
    \label{genvolf} \text{dvol}g^h=\sqrt{|\text{det}(\Gbb_h)|}|\nu_1\wedge...\wedge \nu_d|,
\end{equation} 
where $\Gbb_h$ is the representation matrix of $g^h$ in $(Z_1,...,Z_d)$ and $(\nu_1,...,\nu_d)$ is the dual basis to $(Z_1,...,Z_d)$. 
From this expression, our study will be aimed at the determinant of the matrix $\Gbb_h$. 

\subsection{Expression Of $\Gbb_h^{-1}$}\label{Expression Of G_h^{-1}}
Fix $m$ and the local frame $(Z_1,...,Z_d)$. For $0\leq i\leq r$, denote by $A_i$ the matrix defined such that the 
$j$-$th$ row of $A_i$ is the coefficients of $X^{ij}$ in this frame. Thus the matrix $A_i$ has $N_i$ rows and $d$ columns. 

Denote by $\Gbb_h$ the representation Gram matrix of $g^h$ in this frame. \\ 
The ${A_i}'s$ and $\Gbb_h$ depend on the point $m$ and on the frame. This dependence will be made explicit when needed.
 
\begin{theorem}\label{Tgitoa}
  For all $h>0$, 
we have 

\begin{equation}\label{GitoA}
     \Gbb_h^{-1}=\sum_{i=0}^{r}h^{2i}\,\trans A_i A_i,
     \end{equation} where the matrices $\Gbb_h^{-1}$ and $A_i$ for $i=0,...,r$ have been defined above.
\end{theorem}

\begin{proof}
We represent the map $\sigma_m$ defined in (\ref{sigma}) by the matrix $\Sigma_m$ such that if $U$ is the  
column vector representing $\bfu$ (of size $N$) then 
\[
\Sigma_m U=V \iff \sigma_m (\bfu)\,=\,\sum_{i=1}^d v_iZ_i.
\]

By definition of the matrices $A_i$, we have, as a block matrix,  
\begin{equation}
    \label{raseee}\Sigma_m=\begin{pmatrix}
    \trans A_0& \trans A_1 & ... &\trans A_{\bfr}
\end{pmatrix}.
\end{equation}

Denote by $I^{h}$ the block matrix
$$
I^h=\begin{pmatrix}I_{N_0}&\vline&0&\vline&0&\vline&0\\ \hline0&\vline&h^{2}I_{N_1}&\vline&0&\vline&0\\
  \hline 0&\vline&0&\vline&\ddots&\vline&0\\ \hline 0&\vline&0&\vline&0&\vline&h^{2\bfr}I_{N_{\bfr}} \end{pmatrix},
$$
where for $i=0,...,{\bfr}$, $I_{N_i}$ is the $N_i\times N_i$ identity matrix. Observe that $(I^h)^{-1}=I^{1/h}$, and that, 
if $\bfu$ is represented by $U$ then  
\[
q_h(\bfu)\,=\,\trans U I^{1/h} U. 
\]

Let $X_m \,=\,\sum\limits_{i=1}^d v_i Z_i$. According to the previous section, we need to compute 
$\bfu = ({\tilde{\sigma}_m})^{-1}(X_m)$. Denoting by $U$ the column-vector that represents $\bfu$, 
$U$ is characterized by the following system:
\[
\left \{
\begin{array}{c}
\text{ there exists $V$ such that } \Sigma_m U \,=\, V \\
\text{ if $W$ satisfys } \Sigma_mW=0, \text{ then } \trans W I^{1/h}U  =  0.
\end{array}
\right.
\]

The second line says that $I^{1/h}U$ belongs to the Euclidean orthogonal complement of 
$\ker \Sigma_m$ so that there exists $\tilde{V}$ such that  
\[
I^{1/h}U = \trans{}\Sigma_m \tilde{V},
\]
so that $U=I^h \trans{}\Sigma_m \tilde{V}$. The first line then implies that 
\[
\Sigma_m I^h \trans{}\Sigma_m \tilde{V} = V.
\]
We set $R_{h,m}=\Sigma_m I^h \trans{}\Sigma_m$ and observe that this matrix is invertible  
since $\Sigma_m$ has rank $d$ (this is implied by H\"ormander's condition). We obtain that 
\begin{align*}
g_m^h(X_m) & \,=\,\trans U I^{1/h} U \\
&\,=\, \trans V \trans (I^h\trans{} \Sigma_mR_{h,m}^{-1})I^{1/h} I^h\trans{}\Sigma_m R_{h,m}^{-1}V \\
& \,=\,\trans V R_{h,m}^{-1}\Sigma_m I^h\trans{}\Sigma_m R_{h,m}^{-1}V\\
&\,=\,\trans V R_{h,m}^{-1} V.
\end{align*}

So we have 
\[
\Gbb_h(m) \,=\,R_{h,m}^{-1}.
\]
The claim follows since 
\[
R_{h,m}\,=\,\sum_{i=0}^{{\bfr}}h^{2i}\,\trans A_i A_i
\]

\end{proof}

  \begin{remark}
    This theorem is true for an arbitrary frame on $M$. In particular, it is true for a coordinate frame or an adapted frame (definition \ref{defadfr}).
    The fact that $g_m^h$ is smooth thus also follows from this expression. 
 \end{remark}
 \begin{remark}
   \label{zerovf}   We could include the zero vectors obtained from the iterative brackets of the initial vector fields in the definition of $A_i$. Indeed, 
the matrix $\trans A_i A_i$ will not change by adding a row of zeros for $A_i$. 
 \end{remark}
 
 As a consequence of (\ref{GitoA}), the determinant of $\Gbb_h^{-1}$ is a polynomial in $h$ with non-negative coefficients that are smooth functions of $m$
 (see \cite{friedland2020note}). However, this expression does not provide any information on the limit $h\rightarrow 0$. For instance, neither do we know 
 on the first non-zero coefficient (the leading term in the expansion) nor the power of $h$ in the expansion near $0$.
 In the following section we will prove that the determinant of $\Gbb_h^{-1}$ is asymptotic near $0$ to $f(m)h^ {\varsigma(m)}$ for some $\varsigma(m)$
 to be determined, and then study $f$. 

 We will provide two approaches. The first one is a spectral approach, in which we describe the eigenvalues of $\Gbb_h^{-1}$, 
while the other one depends on the special properties of an adapted frame. 
 
 \subsection{The Spectral Approach}\label{specthapp}
 Fix some coordinates $x=(x_1,...,x_d)$ that are defined on an open set $U=U_x$. We define $A_i$ in the same way as in the previous subsection,
 with respect to $x$. Denote by $\Gbb_h$ the Gram representation matrix of $g^h$ in the coordinate $x$. In these coordinates,
 $\text{dvol}g^h$ can be locally written as
 \begin{equation}
     \label{dvolgs}\text{dvol}g^h=\sqrt{|\text{det}(\Gbb_h)|}|dx_1...dx_d|.
   \end{equation}
   
   We remind the reader that we are working locally. 
The $A_i$'s, $\Gbb_h$ and $\text{dvol}g^h$ depend on the point in the coordinates chosen, so an appropriate notation
   will be introduced when convenient, but for now, when no dependence is shown, 
this means that we work at a fixed point in the fixed coordinate frame. 
  
\subsubsection{Pointwise Limiting Behaviour}
The pointwise behaviour (in the limit $h$ goes to $0$) of $\Gbb_h$ will follow from a general theorem on self-adjoint matrices 
that we now prove.

Let $(A_i)_{i=0,\cdots, {\bfr}}$ be ${\bfr}+1$ self-adjoint $N_i\times d$ matrices (with  $\sum\limits_{i=0}^{\bfr} N_i \defeq N$). 
We aim at describing the spectrum of the family of matrices 
\[
\Acal_h \defeq \sum_{i=0}^{\bfr} h^{2i}\,\trans A_i A_i.
\]
We observe that if we introduce the block-matrix 
\[
B_h\,\defeq\, 
\begin{pmatrix}
A_0\\
hA_1\\
\vdots\\
h^{\bfr} A_{\bfr}
\end{pmatrix} 
\]
then we have 
\[
\Acal_h \,=\, \trans B_h B_h.
\]

Observe that the rank of $B_h$ does not depend on $h$ in $]0,+\infty[$. Since $\rank(\Acal_h)=\rank \,(B_h)$, we obtain that $\rank(\Acal_h)$ also does not depend on $h$ in $]0,+\infty[$ and that $\Acal_h$ is invertible if and only if $B_h$ has rank $d$. 

We define  $\mv_0=\text{Im}(\trans A_0 A_0)$ and let, for $1\leq j\leq {\bfr}$, 
$$\mathcal{V}_{j}=\mv_{j-1}+\text{Im}(\hat{A}_j),$$  
where \begin{equation}
\label{rapoapo}\hat{A}_j=\restr{ \trans A_j A_{j}}{(\mv_{j-1})^\perp},
\end{equation}  
and $\mathcal V_j^\perp$ represents the orthogonal complement of $\mathcal V_j$ with respect to the canonical 
inner product in $\mathbb{R}^d$.  We also set $n_j=\dim \Vcal_j$ for $j=0,\cdots {\bfr}$ and $n_{-1}=0$.

Since $\Acal_h$ is an analytic family of self-adjoint matrices, analytic perturbation theory applies and its spectrum
is organized into real analytic eigenbranches
(see \cite{kato2013perturbation}). With these notations we have the following theorem.

\begin{theorem}\label{thm:specAh}
  Assume $B_h$ has rank $d$ for all $h>0$. Then $n_{\bfr}=d$ and,
  for any $0\leq j\leq {\bfr}$, there are $n_j-n_{j-1}$ eigenbranches $\{\lambda^j_i(h)\}_{1\leq i\leq n_j-n_{j-1}}$ of $\Acal_h$
such that 
$$\lambda_i^j(h)=h^{2j}\eta_i^j(h),$$ where for any $j$, 
the analytic functions $\{\eta_{i}^j(h)\}_{1\leq i\leq n_j-n_{j-1}}$ converges, as $h\rightarrow 0$, to the 
$n_j-n_{j-1}$ non-zero eigenvalues of $\hat{A}_j$.
\end{theorem}

As a corollary of this theorem we can describe the spectrum of $\Gbb_h^{-1}$.

\begin{theorem}\label{ases}
Fix some point $m\in U$. For any $0\leq j\leq \bfr$, there are $n_j(m)-n_{j-1}(m)$ eigenbranches $\{\lambda^j_i(h)\}_{1\leq i\leq n_j-n_{j-1}}$ of $\Gbb_h^{-1}$ 
such that 
$$\lambda_i^j(h)=h^{2j}\eta_i^j(h),$$ where for any $j$, 
the analytic functions $\{\eta_{i}^j(h)\}_{1\leq i\leq n_j-n_{j-1}}$ do not vanish at $h=0$.
\end{theorem}  

We now proceed to the proof of the theorem \ref{thm:specAh}

\begin{proof}
For $1\leq k\leq {\bfr}$, let $(H_k)$ be the following assertion: if we set
    \begin{equation}
        \label{skey}\Acal_k(h)={ \trans A_0} A_0+h^2{\,\trans A_1} A_1+...+h^{2k}{\,\trans A_k} A_k,
    \end{equation}  then, we have that
    \begin{equation}\label{spec}
    \text{spec}(\Acal_k(h))=\{\{h^{2j}\{\rho_i^j(h)\}\}_{0\leq j\leq k, 1\leq i\leq n_j-n_{j-1}},0,...,0\},
  \end{equation}
where, for any $0\leq j\leq k$, the analytic functions $\{\rho_i^j(h)\}_{1\leq i\leq n_j-n_{j-1}}$ converge to the non-zero eigenvalues of $\hat{A}_j.$ 
\\
Moreover, if $\{\zeta_i^j(h)\}_{0\leq j\leq k,1\leq i\leq n_j-n_{j-1}}$ is a set of orthonormal eigenfunctions corresponding to the non-zero eigenvalues of $\Acal_k(h)$,
then
$$\text{span}\{\zeta_i^j(0),0\leq j\leq k,1\leq i\leq n_j-n_{j-1}\}=\mathcal{V}_k.$$


We prove $(H_k)$ by induction on $k$. Since the proof for $k=1$ and the proof of $(H_k)\implies (H_{k+1})$ are completely parallel (see Remark \ref{k=1} below), 
we only prove the latter.

Suppose now that $(H_k)$ holds true. Write the spectrum of $\Acal_k(h)$ as in (\ref{spec}):
\[
  \text{spec}(\Acal_k(h))=\{\{h^{2j}\{\rho_i^j(h)\}\}_{0\leq j\leq k, 1\leq i\leq n_j-n_{j-1}},0,...,0\},
\]
and let $\{\zeta_i^j(h)\}_{0\leq j\leq k,1\leq i\leq n_j-n_{j-1}}$ be the corresponding analytic branches of eigenfunctions corresponding to the non-zero eigenvalues of 
$\Acal_k(h)$.

Denote by $\mathcal{J}_k(h)$ the subspace
$$
\mathcal{J}_k(h):=\text{span}\{\zeta_i^j(h), {0\leq j\leq k,1\leq i\leq n_j-n_{j-1}}\}.
$$

Using the induction hypothesis, we have:
  \begin{equation}\label{stepk}
      \mathcal{J}_k(0)=\mathcal{V}_k.
    \end{equation}

We now prove $(H_{k+1})$. 

First we observe that the rank of $\Acal_{k+1}(h)$ is independent of $h\neq 0$. It follows that there is a fixed number of 
$0$ eigenbranches. 

Then, using analyticity and the fact that $\Acal_{k+1}(h) - \Acal_{k}(h) = O(h^{2k+2})$, we can write
    \begin{equation}
    \label{habahab}\text{spec}(\Acal_{k+1}(h))=\{\{h^{2j}\Tilde{\rho}_i^j(h)\}_{0\leq j\leq{k},1\leq i\leq n_j-n_{j-1}},\{h^{2k+2}\Tilde{\rho}_i^{k+1}(h)\}_{1\leq i\leq d-n_k}\},
\end{equation} 
with 
\begin{equation}\label{fev}
    |h^{2j}(\Tilde{\rho}_i^j(h)-{\rho}_i^j(h))|=O(h^{2k+2}),
\end{equation}
 and  
\begin{equation}\label{fef}
     \Tilde{\zeta}_i^j(0)=\zeta_i^j(0),
 \end{equation}
 for all $0\leq j\leq k$ and $1\leq i\leq n_k-n_{k-1}$, where

 \begin{equation}
     \label{popopopopopopopo}\left\{\{\tilde{\zeta}_i^j(h)\}_{0\leq j\leq k,1\leq i\leq n_j-n_{j-1}
},\{\tilde{\zeta}_i^{k+1}(h)\}_{1\leq i\leq d-n_k}\right\}
\end{equation}
is a set of orthonormal eigenfunctions corresponding to 
the eigenvalues in (\ref{habahab}). 

This implies, by the induction hypothesis, that $(\Tilde{\rho}_i^j(0),\tilde{\zeta}_i^j(0))$ is an eigenpair of $\hat{A}_j$ for any $0\leq j\leq k,1\leq i\leq n_j-n_{j-1}$. \\ 

Fix $i_1\in \{1,...,d-n_k\}$ and consider $\Tilde{\rho}_{i_1}^{k+1}(h)$ with corresponding eigenfunction $\Tilde{\zeta}_{i_1}^{k+1}(h)$. 
The equations (\ref{stepk}) and (\ref{fef}) imply that 
$\{\tilde{\zeta}_i^j(0)\}_{0\leq j\leq k,1\leq i\leq {n_j-n_{j-1}}}$ is an orthonormal basis of $\mathcal{V}_k.$ So, we get  
\begin{equation}\label{me3}
    \Tilde{\rho}_{i_1}^{k+1}(h)=O(h^{2k+2}) \text{ and }\Tilde{\zeta}_{i_1}^{k+1}(0)\in \mathcal{V}_k^\perp\;\;\; \text{as } h\rightarrow 0.
\end{equation}
Let $\mathcal{P}_h^k$ and $\mathcal{Q}_h^k$ be the orthogonal projections on $\mathcal{J}_k(h)$ and $(\mathcal{J}_k(h))^\perp$ respectively. Write 

$$\Tilde{\zeta}_{i_1}^{k+1}(h)=\mathcal{P}_h^k\Tilde{\zeta}_{i_1}^{k+1}(h)+\mathcal{Q}_h^k\Tilde{\zeta}_{i_1}^{k+1}(h).$$ 

The projectors $\mathcal{P}_h^k$ and $\mathcal{Q}_h^k$ depend analytically on $h$, and, by (\ref{me3}), we have that
\begin{equation}
     \label{walabatal}\mathcal{P}_h^k\Tilde{\zeta}_{i_1}^{k+1}(h)=o(1) \text{ and }  \mathcal{Q}_h^k\Tilde{\zeta}_{i_1}^{k+1}(h)=\Tilde{\zeta}_{i_1}^{k+1}(0)+o(1),
   \end{equation}
   as $h\rightarrow0$. 

Therefore, the eigenvalue equation implies that
   \begin{equation}\label{LE}
    h^{2k+2}\Tilde{C}_h\Tilde{\zeta}_{i_1}^{k+1}(h)+h^{2k+2}\Tilde{B}_h\Tilde{\zeta}_{i_1}^{k+1}(h)=h^{2k+2}\Tilde{\rho}_{i_1}^{k+1}(h)\mathcal{Q}_h^k\Tilde{\zeta}_{i_1}^{k+1}(h),
  \end{equation}
  where $$\Tilde{C}_h=\mathcal{Q}_h^k\trans A_{k+1}A_{k+1}\mathcal{P}_h^k$$ and 
$$\Tilde{B}_h=\mathcal{Q}_h^k \trans A_{k+1}A_{k+1}\mathcal{Q}_h^k.$$ 

Equation (\ref{LE}) gives, after dividing by $h^{2k+2}$ and writing $\tilde{C_h}\Tilde{\zeta}_{i_1}^{k+1}(h)=o(1)$ (by (\ref{walabatal})), the following equation 
$$o(1)+\Tilde{B_h}\Tilde{\zeta}_{i_1}^{k+1}(h)\,=\,\Tilde{\rho}_{i_1}^{k+1}(h)\mathcal{Q}_h^k\Tilde{\zeta}_{i_1}^{k+1}(h).$$ 

Evaluating at $h=0$ we obtain  
\begin{equation}\label{eve}
    (\tilde{Q}_0\trans A_{k+1}A_{k+1}\tilde{Q}_0-\Tilde{\rho}_{i_1}^{k+1}(0))\tilde{Q}_0\Tilde{\zeta}_{i_1}^{k+1}(0)=0.
  \end{equation}
  Therefore, (\ref{stepk}) and (\ref{eve}) imply that $(\Tilde{\rho}_{i_1}^{k+1}(0),\Tilde{\zeta}_{i_1}^{k+1}(0))$ is an eigenpair of $\hat{A}_{k+1}$.\\

Finally, the orthonormality of the eigenbranches in (\ref{popopopopopopopo}) implies the orthonormality of $\{\Tilde{\zeta}_i^{k+1}(0)\}_{1\leq i\leq d-n_k}$. 
Therefore, the $d-n_k$ orthonormal (thus linearly independent) vectors $\{\Tilde{\zeta}_i^{k+1}(0)\}_{1\leq i\leq d-n_k}$ form an eigenbasis of $\mathcal{V}_k^{\perp}$ 
(which is of dimension $d-n_k$ by (\ref{stepk})). So, they cover all the eigenvalues of $\hat{A}_{k+1}$.\\

It remains to prove that the span of the eigenvectors that correspond to the non-zero eigenvalues of $\hat{A}_{k+1}(h)$ span 
$\mathcal{V}_{k+1}$ at $h=0$. We write the eigenvectors in (\ref{popopopopopopopo}) as follows
  \begin{equation}
    \left\{\{\tilde{\zeta}_i^j(h)\}_{0\leq j\leq k,1\leq i\leq n_j-n_{j-1}
},\{\tilde{\zeta}_i^{k+1}(h)\}_{1\leq i\leq n_k-n_{k-1}},\{\til{\zeta}_i^z\}_{1\leq i\leq d-n_{k+1}}\right\},
 \end{equation}
where $\{\til{\zeta}_i^z\}_{1\leq i\leq d-n_{k+1}}$ corresponds to the 0 eigenvalue. 
Again, (\ref{stepk}) and  (\ref{fef}) implies that $\{\tilde{\zeta}_i^j(0)\}_{0\leq j\leq k,1\leq i\leq n_j-n_{j-1}
}$ span $\mathcal{V}_k$. Now, we prove that $\{\tilde{\zeta}_i^{k+1}(0)\}_{1\leq i\leq n_{k+1}-n_{k}}$ span $\text{Im}(\hat{A}_k).$ 
Since $\{\tilde{\zeta}_i^{k+1}(0)\}_{1\leq i\leq n_{k+1}-n_{k}}$ are eigenvectors of $\hat{A}_{k+1}$ it is clear that 
$$\text{span}\{\tilde{\zeta}_i^{k+1}(0),1\leq i\leq n_k-n_{k-1}\}\subset \text{Im}(\hat{A}_{k+1}).$$ 

Now, let $u\in\text{Im}(\hat{A}_{k+1})$. Then, there is $v\in\mathcal{V}_k^\perp$ such that $u=\hat{A}_{k+1}(v).$ We proved that 
$$
\text{span} \left\{\{\tilde{\zeta}_i^{k+1}(0)\}_{1\leq i\leq n_k-n_{k-1}},\{\til{\zeta}_i^z\}_{1\leq i\leq d-n_{k+1}}\right\}=\mathcal{V}_k^{\perp}.
$$ 
So, we can write $v$ as a linear combination 
$$
v=\sum_{i=1}^{n_{k+1}-n_k}a_i\tilde{\zeta}_i^{k+1}(0)+\sum_{i=1}^{d-n_{k+1}}b_i\til{\zeta}_i^z.
$$ 
Then, because $\{\til{\zeta}_i^z\}_{i=1,...,d-n_{k+1}}$ correspond to the zero eigenvalues, we get  

\begin{equation*}
    \begin{split}
        u&=\hat{A}_{k+1}(v)\\
&=\sum_{i=1}^{n_{k+1}-n_k}a_i\hat{A}_{k+1}(\tilde{\zeta}_i^{k+1}(0))+\sum_{i=1}^{d-n_{k+1}}b_i\hat{A}_{k+1}(\til{\zeta}_i^z)\\
&= \sum_{i=1}^{n_{k+1}-n_k}a_i\til\rho_i^{k+1}(0)(\tilde{\zeta}_i^{k+1}(0))\in \text{span}\{\tilde{\zeta}_i^{k+1}(0),1\leq i\leq n_k-n_{k-1}\}.
    \end{split}
\end{equation*}
    Therefore, we get   
$$
\text{span} \left\{\{\tilde{\zeta}_i^j(h)\}_{0\leq j\leq k,1\leq i\leq n_j-n_{j-1}
},\{\tilde{\zeta}_i^{k+1}(h)\}_{1\leq i\leq n_k-n_{k-1}}\right\}=\mathcal{V}_{k+1}.
$$
We conclude $(H_{k+1})$. \\

The induction stops at $k={\bfr}$ and since $n_{\bfr}=d$, we get the final result.
    \end{proof}

\begin{remark}\label{k=1}
The case $k=1$ follows from the same arguments. It is actually simpler because, in that case, the projectors $\Pcal_h$ and $\Qcal_h$ actually do not depend on $h$. 
\end{remark}

    \begin{remark}
    The relevant part of the induction in the proof of theorem \ref{ases} is for $0\leq j\leq {\bfr}(m)$. For $j>{\bfr}(m),\, n_j(m)=d$ and so the statement of theorem indicates that there isn't any  
eigenbranches of order $h^{2j}$. This is because at $m$, $\Acal_{{\bfr}(m)}(h)$ defined by (\ref{skey}) generates all the $d$ non-zero eigenbranches of $G_h^{-1}(m)$, and 
with any further perturbation of higher order, the $d$ new eigenbranches will only be perturbed by a big O of this higher order, but the behavior near $h=0$ remains the same.
\end{remark}

Theorem \ref{ases} implies that the spectrum of $G_h^{-1}$ can be written as follows:
    $$
\{\underbrace{\text { Order 0 terms} }_{n_0(m)},\underbrace{\text { Order $2$ terms} }_{n_1(m)-n_0(m)}...\underbrace{\text { Order 2\textbf{r} terms} }_{d-n_{{\bfr}-1}(m)}\}.
$$ 
In particular, as the determinant of a matrix is the product of the eigenvalues, we get
\begin{corollary}\label{aham}
  For a fixed $m\in U$, the determinant of $\Gbb_h^{-1}$ has the following expansion 
$$
\mathrm{det}(\Gbb_h^{-1}(m))=f_h(m)h^{2\varsigma(m)},
$$ where 
\begin{equation}
        \label{varsigmaaa}\varsigma(m)=\sum_{i=1}^{{\bfr}} i[n_i(m)-n_{i-1}(m)],
    \end{equation}  
and $f_h(m)$ is given by 
$$
f_h(m)=\prod_{j=0}^{\bfr}\prod_{i=1}^{n_j-n_{j-1}}\eta_i^j(h)(m),
$$ 
with $\eta_i^j(h)(m)$ being introduced in theorem \ref{ases} (with the convention that $\prod_{i=1}^0=1$).
    
    Moreover, $f_h(m)$ converges, as $h\rightarrow 0$, to 
    $f(m)\neq 0$, where $f(m)$ is the product of the non-zero eigenvalues of $\hat{A}_i(m)$ (given by (\ref{rapoapo})) for $i=0,...,{\bfr}$.
\end{corollary}

    \subsubsection{Dependence With Respect To The Point}
    In the previous section, we obtained the expansion near $h=0$ for $\text{det}(\Gbb_h^{-1})$ at a fixed point $m\in U$. Theorem \ref{ases} is pointwise and 
the behavior in corollary \ref{aham} is a pointwise behavior. Now, we allow $m$ to move in $M$ and we study the characteristics of the function $f(m)$, 
that was defined in corollary \ref{aham}. In order to obtain a well-defined function, it is convenient to introduce a fixed non-vanishing volume form for 
reference. We choose the volume form associated with $h=1$. We define the function $f_h$ as the ratio $(\mathrm{dvol_{g^1}}/\mathrm{dvol_{g^h}})^2$. 
In any chosen frame, $f_h$ is obtained as $\text{det}(\Gbb_h^{-1})$ multiplied by a non-vanishing smooth function. In particular 
we already know that $f_h$ is a polynomial in $h$ with non-negative smooth coefficients :

\begin{equation}
     \label{finanstntntn}f_h\,=\,\sum_{k\geq0}a_k(m)h^{2k},
\end{equation}
where the $a_k$ are smooth non-negative functions on $M$.

The following proposition is a rephrasing of the preceding section.

\begin{proposition}\label{prop:finitelimit}
For any $m\in M$, let $\varsigma(m) = \sum\limits_{i=1}^{{\bfr}} i[n_i(m)-n_{i-1}(m)]$ then $h^{-2\varsigma(m)}f_h(m)$ converges to a finite positive limit. 
\end{proposition}

We first observe that $\varsigma(m)$ is closely related to the Hausdorff dimension of the sub-Riemannian structure at $m$.

We recall the sub-Riemannian definition of the singular set and define the set $\mathcal{Z}_U=U\cap\mathcal{Z}$. 
        
Setting $n_{-1}(m)=0$, the Hausdorff dimension at $m$ is given by 
$$
Q(m)=\sum_{i=0}^{\bfr}(i+1)[n_i(m)-n_{i-1}(m)],
$$ 
(see \cite{mitchell1985carnot}). 

It follows that
\begin{equation*}
  Q(m)-\varsigma(m)=\sum_{i=0}^{{\bfr}}(i+1)[n_i(m)-n_{i-1}(m)]-\sum_{i=1}^{{\bfr}}i[n_i(m)-n_{i-1}(m)]=d.
\end{equation*}


The dimension $Q$ is constant over any connected component of the equiregular region. 
From now on, we will assume that this constant also does not depend on the component so that there exists 
$\bfQ$ such that 
\[
\begin{split}
Q(m) = \bfQ &\iff m \in M\setminus \Zcal\\
Q(m)\,>\, \bfQ & \iff m\in \Zcal.
\end{split}
\]

We now define the function
\[
  f \defeq \lim_{h\rightarrow 0} h^{2(d-\bfQ)} f_h,
\]
and we prove the following proposition. 

\begin{proposition}\label{singsetoo}
  Under the assumption that the Hausdorff dimension of the (possibly not connected) equiregular region is constant then
  the function $f$ is smooth, positive in the equiregular region and vanishes on the singular region $\mathcal{Z}_U$.
\end{proposition}

\begin{proof}
  The positivity and vanishing are direct consequences of Proposition \ref{prop:finitelimit}, the fact that $Q(m)> \bfQ$ if and only if
  $m$ is in the singular region and the expansion (\ref{finanstntntn}). The latter also shows that the coefficient $a_k$
  identically vanishes is $k< \bfQ-d$ and that $f=a_{\bfQ-d}$. This proves the smoothness. 
\end{proof}

We can now define $\dvolsr = f^{-\frac{1}{2}}\,\mathrm{dvol}{g^1}$. This is a smooth volume form in the equiregular region that blows up near the singular one. 

By construction, we have  
\[
  \dvolsr \,=\,\lim_{h\rightarrow 0} h^{\bfQ-d} \mathrm{dvol}{g^h}.
\]

\begin{remark}
Assume that the singular set is stratified by submanifolds along which the growth vector is constant. Then, 
following the same strategy, it should be possible to define a natural volume form associated with the approximation scheme on any stratum.
\end{remark}

\subsection{Second Approach: Adapted Frame}\label{rlnwzpopp}
In the preceding section, we defined a volume form that is adapted to the approximation scheme by working in any local frame. Instead, in this section, we will
borrow the notion of adapted basis from \cite{barilari2013formula}. This will enable us to compare our volume form to Popp's volume form.   

We first define an adapted basis.
\begin{definition}\label{defadfr}
  We say a local frame $Z_1,...,Z_d$ is adapted at the point $m$, if, in the vicinity of $m$, $Z_1,...,Z_{n_i}$ is a local frame for $D_i$, for any $0\leq i\leq \bfr$,
  where $n_i=\text{dim}(D_i(m)),$ and $Z_1(m),...,Z_{n_0}(m)$ are orthonormal. 
\end{definition}

\subsubsection{Recovering Corollary \ref{aham}}
Fix $m$ and some adapted frame $Z_1,...,Z_d$ at $m$. As before, let $A_{(i)}$ be the $N_i\times d$ matrix defined such that the $j$-$th$ row of $A_{(i)}$
is the coefficients of $X^{ij}$ in the adapted frame. 
 
We describe the matrix $A_{(i)}$ for $0\leq i\leq \bfr$. The rows of the matrix $A_{(i)}$ are the vectors $X^{ij}$ for $1\leq j\leq N_i$ written in the adapted frame $Z_1,...,Z_d$. For each $1\leq j\leq N_i$, the vector $X^{ij}$ is in $D_i$ which has $\{Z_1,...,Z_{n_i}\}$ as a basis. Thus, the last $d-n_i$ columns of the matrix $A_{(i)}$ are zero columns. More precisely, if we denote by $a^k_{ij}$ the coefficient of $Z_k$ in the expression of $X^{ij}$;
$$
X^{ij}=\sum_{k=1}^d a^k_{ij}Z_k,
$$ 
then $a^k_{ij}=0$ for all $k>n_i$ and $A_{(i)}$ has the following expression 
\begin{equation}
    \label{characofa_(i)}A_{(i)}=
\begin{pmatrix}
    a^1_{i1}& ... & a^{n_i}_{i1}&0&...&0 \\
    a^1_{i2}& ... & a^{n_i}_{i2}&0&...&0 \\
    \vdots&\ddots&\vdots &0&...&0 \\
    a^1_{iN_i}& ... & a^{n_i}_{iN_i}&0&...&0
\end{pmatrix}.
\end{equation} 
Now, for any $0\leq i\leq \bfr$ and $1\leq j\leq N_i$ define the coefficients $\bar{a}^k_{ij}$ as follows: 
$$
\bar{a}^k_{ij}=\begin{cases}
    0 & 1\leq k\leq n_{i-1} \\
    a^k_{ij} & n_{i-1}+1\leq k\leq n_i
\end{cases},
$$ 
and define the matrix $\bar A_{(i)}:=(\bar{a}^k_{ij})_{1\leq j\leq N_i, 1\leq k\leq n_i}$ with $N_i$ rows and $n_i$ columns 
($\bar A_{(i)}$ denotes the representation matrix of $X^{ij}$ in $D_i$ mod $D_{i-1}$, letting all the coefficients that correspond to $Z_1,...,Z_{n_i}$ equal $0$). 

We introduce the matrix $M_i$ as the nonzero block of $\trans{\bar A_{(i)}}\bar A_{(i)}$ so that we have the following diagram.

\begin{equation*}
 \trans {\bar A_{(i)}} \bar A_{(i)}=
   \begin{pNiceMatrix}[margin] 
0&0&0&0 \\
\Block[draw,fill=red!15,rounded-corners]{1-4}{}
*&*&*&* \\
0&0&0&0 \\
\end{pNiceMatrix}
\begin{pNiceMatrix}[margin]
0 & \Block[draw,fill=red!15,rounded-corners]{4-1}{}* & 0  \\
0 & * & 0  \\
0 &   * & 0  \\
0 & * & 0 \\
\end{pNiceMatrix}=\begin{pNiceMatrix}[margin] 
0&0&0 \\
0&\Block[draw,fill=blue!15,rounded-corners]{1-1}{}*&0 \\
0&0&0 \\
\end{pNiceMatrix}.
\end{equation*}

The red-shaded region in $\bar A_{(i)}$ corresponds to the vectors $Z_{n_{i-1}+1},...,Z_{n_i}$, and $M_i$ is the blue-shaded sub-matrix.

\begin{proposition}\label{recovv}
 Fix $m\in M$.  If we denote by $\til{\Gbb}_h$ the representation Gram matrix of $g^h$ with respect to the adapted frame $Z_1,...,Z_d$, then, near $h=0$, we have  

\begin{equation}
        \label{analogue}\mathrm{det}(\til{\Gbb}_h^{-1})\sim h^{2\varsigma(m)}\prod_{i=1}^{\bfr}\mathrm{det}(M_i).
    \end{equation}
  \end{proposition}
  
\begin{proof}  

By theorem \ref{Tgitoa}, We have that 
\begin{equation}\label{GitoA2}
     {\til{\Gbb}_h}^{-1}=\sum_{i=0}^{\bfr}h^{2i}\,\trans A_{(i)} A_{(i)}.
     \end{equation}

Using the characteristics of the matrix $A_{(i)}$ (description (\ref{characofa_(i)})), formula (\ref{GitoA2}) implies that 

\begin{equation}\label{expoftilg}   
    \til {\Gbb}_h^{-1}= \left(
    \begin{array}{c|c|c|c|c}
      M_0+O(h^2) & O(h^2) &...&O(h^{2({\bfr}-1)})& O(h^{2{\bfr}})\\
      \hline
      O(h^2) & h^2M_1+O(h^4) &O(h^4)&...&\vdots\\
      \hline 
        O(h^4)& O(h^4)&h^4M_2+O(h^6)&...&\vdots\\
      \hline
      \vdots&...&...&\ddots&O(h^{2{\bfr}}) \\
      \hline
     O(h^{2{\bfr}})&...&...&O(h^{2{\bfr}})& h^{2r}M_{\bfr}
    \end{array}
    \right).
\end{equation} 

Let  
$$
G_{(1)}(h)=\left(
    \begin{array}{c|c|c|c|c}
      M_0 & 0 &...&0& 0\\
      \hline
      0 & h^2M_1 &0&...&\vdots\\
      \hline 
        0& 0&h^4M_2&...&\vdots\\
      \hline
      \vdots&...&...&\ddots&0 \\
      \hline
     0&...&...&0& h^{2{\bfr}}M_{\bfr}
    \end{array}
    \right),$$ 

and 
$$
G_{(2)}(h)=\til {\Gbb}_h^{-1}-G_{(1)}(h)=\left(
    \begin{array}{c|c|c|c|c}
     O(h^2) & O(h^2) &...&O(h^{2({\bfr}-1)})& O(h^{2{\bfr}})\\
      \hline
      O(h^2) & O(h^4) &O(h^4)&...&\vdots\\
      \hline 
        O(h^4)& O(h^4)&O(h^6)&...&\vdots\\
      \hline
      \vdots&...&...&\ddots&O(h^{2{\bfr}}) \\
      \hline
     O(h^{2{\bfr}})&...&...&O(h^{2{\bfr}})& 0
    \end{array}
    \right).$$ 

    Moreover, we have that $M_i$ is invertible, and so $G_{(1)}(h)$ also is. 
    Indeed, $M_i$ is an $n_i-n_{i-1}$ matrix with 
$$
\text{rank}(M_i)=\text{rank}\left(\trans {\bar A_{(i)}} \bar A_{(i)}\right)=\text{rank}\left(\bar A_{(i)}\right).
$$ 

As $n_i\leq N_i$, then 
$$\text{rank}\left(\bar A_{(i)}\right)\leq n_i-n_{i-1}.$$ 

Since the frame is adapted, $\{Z_{n_{i-1}+1},...,Z_{n_i}\}$ is a basis of a linear complement for $D_{i-1}$ in $D_i$, then 
$\text{rank}(\bar A_{(i)})= n_i-n_{i-1}$. 

Thus, $\text{rank}(M_i)=n_i-n_{i-1}$, and so it has full \text{rank} (in particular, it is invertible).\\

Therefore, we have 
\begin{equation}
        \label{analogto} 
\begin{split}
            \text{det}(\til {\Gbb}_h^{-1})&=\text{det}(G_{(1)}(h)(I_{d}+(G_{(1)}(h))^{-1}G_{(2)}(h)))\\
&=\text{det}(G_{(1)}(h))\text{det}(I_{d}+(G_{(1)}(h))^{-1}G_{(2)}(h)).
        \end{split} 
    \end{equation} 

Now, as $(G_{(1)}(h))^{-1}G_{(2)}(h)$ is simply multiplying the $i^{th}$ row of $G_{(2)}(h)$ by $h^{-2i}M_i^{-1}$ for any $0\leq i\leq {\bfr}$, 
then, the limit of $(G_{(1)}(h))^{-1}G_{(2)}(h)$, which we denote by $\bar G$, is a strictly lower triangular matrix. 


It follows that, near $h=0$ 
\begin{equation}\label{ssddssdd}
        \text{det}(I_{d}+(G_{(1)}(h))^{-1}G_{(2)}(h))\sim 1,
      \end{equation}
so that 
      \begin{equation}
        \label{lkaymta}\text{det}(\til\Gbb_h^{-1})\sim \text{det}(G_{(1)}(h))=h^{2\varsigma(m)}\prod_{i=1}^{\bfr}\text{det}(M_i).
    \end{equation}
\end{proof}

\begin{remark}
  Proposition \ref{recovv} recovers corollary \ref{aham} directly. It can however be remarked that the spectral approach gives more information 
that could be useful in order to obtain finer properties of the approximating Riemannian structure.
\end{remark}

\subsubsection{Comparison To Popp's Volume}
We first recall the definition of Popp's volume following \cite{barilari2013formula}, where the formula of Popp's volume is given in terms of any adapted frame
of the tangent bundle. We then compare our volume form induced by the approximation scheme to Popp's volume.
 
Fix an adapted local frame $Z_1,...,Z_d$. For $j=1,...,{\bfr}$, we define the adapted structure constants $b_{i_1...i_j}^l\in\mathcal{C}^\mathcal{1}(M)$ as follows:

\begin{equation}
     \label{bi's} [Z_{i_1},[Z_{i_2},...,[Z_{i_{j}},Z_{i_{j+1}}]...]]=\sum_{l=n_{j-1}+1}^{n_j}b_{i_1i_2...i_j}^lZ_l\;\text{mod} D^{j-1},  
 \end{equation} where $1\leq i_1,...,i_j\leq n_0$. We define the $n_j-n_{j-1}$ dimensional square positive definite matrix $B_j$ as follows \begin{equation}
     \label{B_j's} [B_j]^{hl}=\sum_{i_1,...,i_j=1}^{n_0} b_{i_1i_2...i_j}^hb_{i_1i_2...i_j}^l, \;\;\; j=0,...,m,
\end{equation} where $B_0$ is the $n_0\times n_0$ identity matrix.
 
 \begin{definition}[Barilari-Rizzi\cite{barilari2013formula}]
     Let $\nu_1,...,\nu_d$ be the dual frame to $Z_1,...,Z_d$. Then, Popp's volume is given by 
\begin{equation}
         \label{popbari} d\mathcal{P}=\dfrac{1}{\sqrt{\prod_j\text{det}(B_j)}}\nu_1\wedge...\wedge\nu_d.
     \end{equation}
\end{definition} 
Hereafter, $\mathcal{P}$ will always denote the Popp's volume.

In order to compare our volume to Popp's, we need to relate these structure coefficients and the matrices $M_i$.
We begine with $M_0$.

\begin{proposition}\label{krononono}
The matrix $M_0$ defined with respect to the adapted frame $Z_1,...,Z_d$ is the identity matrix; that is $M_0=I_{n_0}$. 
\end{proposition}
        
\begin{proof}
     Expansion (\ref{ssddssdd}) implies that for $h$ small enough, $I_{d}+(G_{(1)}(h))^{-1}G_{(2)}(h)$ is invertible. Since
        \begin{equation}
            \label{tkngtheinv} \til {\Gbb}_h^{-1}=G_{(1)}(h)(I_{d}+(G_{(1)}(h))^{-1}G_{(2)}(h)),
        \end{equation}
              we obtain that 
$$
\til{\Gbb}_hG_{(1)}(h)=(I_{d}+(G_{(1)}(h))^{-1}G_{(2)}(h))^{-1}.$$
Thus, 
$$
M_0=\restr{\til{\Gbb}_hG_{(1)}(h)}{D_0}=\restr{(I_{d}+(G_{(1)}(h))^{-1}G_{(2)}(h))^{-1}}{D_0}.
$$
 Since the limit of $I_{d}+(G_{(1)}(h))^{-1}G_{(2)}(h)$ (that we denote by $I_d+\Bar{G}$) exists and is invertible, we get that 
$$
M_0=\restr{(I_{d}+\Bar{G})^{-1}}{D_0}.
$$

Now, as $\Bar{G}$ is a strictly lower triangular matrix, it is nilpotent, say of index $n$. So, 
$$
(I_{d}+\Bar{G})^{-1}=I_{d}+\sum_{k=1}^{n-1}(-1)^k\Bar{G}^k.
$$ 
Therefore, since for any $k$, $\restr{\Bar{G}^k}{D_0}=0$ we get that $M_0=I_{n_0}$. 
\end{proof}

Now we compare the matrices $M_i$ and $B_i$ defined by \ref{B_j's}.  
  \begin{theorem}
    Denote by $(\mu_{(i),\kappa_1\kappa_2})_{1\leq \kappa_1,\kappa_2\leq n_i-n_{i-1}}$ the entries of the matrix $M_i$.
    Then, for every $1\leq i\leq {\bfr}$, and any $1\leq \kappa_1,\kappa_2\leq n_i-n_{i-1}$, we have
    $$
    \mu_{(i),\kappa_1\kappa_2}=[B_i]^{\kappa_1\kappa_2},
    $$
    where $[B_i]^{\kappa_1\kappa_2}$ are the entries of the matrix $B_i$ defined by (\ref{B_j's}).
  \end{theorem}
  
        \begin{proof}
          Denote by $J:=\{1,...,N_0\}$. For any $j\in J$, write $X^{0j}=\sum_{l=1}^{n_0}a_{(0)j}^lZ_l$. Then, for $j_1,...,j_{i+1}\in J$, we have
          \begin{equation}
                \label{[x_j,x_k]} [X^{0j_1},[X^{0j_2}...[X^{0j_{i}},X^{0j_{i+1}}]...]=\sum_{l_1,...,l_{i+1}=1}^{n_0}a_{(0){j_1}}^{l_1}...a_{(0){j_{i+1}}}^{l_{i+1}}[Z_{l_1},[Z_{l_2},...[Z_{l_i},Z_{l_{i+1}}]...].
              \end{equation}
              Then, for any $1\leq i\leq {\bfr}$, any $n_{i-1}<n\leq n_i$, and any $(j_1,...,j_{i+1})\in J^{i+1}$, we get that the coefficient of $[X^{0j_1},[X^{0j_2}...[X^{0j_{i}},X^{0j_{i+1}}]...]$ in $Z_n\in D_i$ mod $D_{i-1}$ is given by
              \begin{equation}
                \label{a_1ij^n} a_{(i),(j_1,...,j_{i+1})}^n=\sum_{l_1,...,l_{i+1}=1}^{n_0}a_{(0){j_1}}^{l_1}...a_{(0){j_{i+1}}}^{l_{i+1}}b_{l_1,...,l_{i+1}}^n,
              \end{equation}
              where the $b's$ are the adapted structure constants defined in (\ref{bi's}). Denote by $\til J:=\{1,...,n_0\}$.
              Therefore, for any $1\leq i\leq {\bfr}$ and any $1\leq \kappa_1,\kappa_2\leq n_i-n_{i-1}$, we compute
              \begin{equation}\label{ttaretzida}
                \begin{split}
                  \mu_{(i),\kappa_1\kappa_2}&=\sum_{(j_1,\cdots,j_{i+1})\in {\til J}^{i+1}}a_{(i),(j_1,...,j_{i+1})}^{\kappa_1}a_{(i),(j_1,...,j_{i+1})}^{\kappa_2}\\
                  &=\sum_{\underset{m_1,...,m_{i+1}\in\til J}{l_1,...,l_{i+1}\in\til J}}
                  \sum_{{(j_1,...,j_{i+1})\in {\til J}^{i+1}}}a_{(0){j_1}}^{l_1}...a_{(0){j_{i+1}}}^{l_{i+1}}b_{l_1,...,l_{i+1}}^{\kappa_1}a_{(0){j_1}}^{m_1}...a_{(0){j_{i+1}}}^{m_{i+1}}b_{m_1,...,m_{i+1}}^{\kappa_2}\\
                  &=\sum_{\underset{m_1,...,m_{i+1}\in\til J}{l_1,...,l_{i+1}\in\til J}}\left(\sum_{j_1\in \til J}a_{(0){j_1}}^{l_1}
a_{(0){j_1}}^{m_1}\right)...\left(\sum_{j_{i+1}\in \til J}a_{(0){j_{i+1}}}^{l_{i+1}}
a_{(0){j_{i+1}}}^{m_{i+1}}\right)b_{l_1,...,l_{i+1}}^{\kappa_1}b_{m_1,...,m_{i+1}}^{\kappa_2}\\
&=\sum_{\underset{m_1,...,m_{i+1}\in\til J}{l_1,...,l_{i+1}\in\til J}}\delta_{l_1}^{m_1}...\delta_{l_{i+1}}^{m_{i+1}}b_{l_1,...,l_{i+1}}^{\kappa_1}b_{m_1,...,m_{i+1}}^{\kappa_2}\\
&=\sum_{l_1,...,l_{i+1}\in\til J}b_{l_1,...,l_{i+1}}^{\kappa_1}b_{l_1,...,l_{i+1}}^{\kappa_2}=[B_i]^{\kappa_1\kappa_2},
                \end{split}
            \end{equation}
         where the Kronecker deltas are due to proposition \ref{krononono}.
\end{proof}

    Recall that $\mathcal{P}$ denotes the Popp's volume. Comparing (\ref{lkaymta}) and (\ref{popbari}), we deduce the following corollary.
    \begin{corollary}
        If we denote by $\mathcal{P}_o$ the volume obtained from our approximation scheme, then we have \begin{equation}
            \label{relationbpvaov} \mathcal{P}=\mathcal{P}_o.
        \end{equation}
    \end{corollary}
    We give a typical example, standard in the context of an equiregular sub-Riemannian structure.
    \begin{example}[The Heisenberg Case on $\mathbb{R}^3$]
         On $\mathbb R^3$, consider the smooth vector fields $$X^{01}=\dl_x \text{ and } X^{02}=\dl_y+x\dl_z.$$ For any $x\in \mathbb R^3$, $X^{01},X^{02} \textit{ and } X^{11}:=[X^{01},X^{02}]=\dl_z=:-X^{12}$ span $\mathbb R^3$ (equiregular case of step 2). Following definition (\ref{g^inf}), define the sub-Riemannian metric $g^0$ on $\mathbb R^3$ as $$g^0(X)=\inf\{u_1^2+u_2^2; {\bfu}=(u_i)_{i=1,2}\in\mathbb R^2, u_1X^{01}+u_2X^{02}=X\}.$$ Following definition (\ref{g^s}), define the Riemannian metric $g^h$ on $\mathbb R^3$ as $$g^h(X)=\inf\{u_1^2+u_2^2+h^{-2}u_3^2+h^{-2}u_4^2; {\bfu}=(u_i)_{i=1,2,3,4}\in\mathbb R^4, u_1X^{01}+u_2X^{02}+u_3X^{11}+u_4X^{12}=X\}.$$ For $X=(a,b,c)\in\mathbb R^3$, direct computation implies that \begin{equation}\label{ltdisc}
             g^h((a,b,c))=a^2+b^2+\dfrac{(c-xb)^2}{2h^2},
         \end{equation} which is clearly a smoothly defined inner product on $\mathbb R^3$. Moreover, for any $(a,b,c)\in \mathbb R^3$, $$g^0((a,b,c))=a^2+b^2.$$ For $g^0$, which is the limit of $g^h$ as $h\rightarrow 0$, to be finite, $c-xb$ must be $0$, which means that $X=aX^{01}+bX^{02}$, that is $X$ is horizontal.
       
        Denote by $\Gbb_h$ the representation matrix of $g^h$. We compute  $$A_0=\begin{pmatrix}
1 & 0 & 0\\
0 & 1 & x
\end{pmatrix}, A_1=\begin{pmatrix}
0 & 0 & 1 \\
0 &0  &-1
\end{pmatrix}\text{ and } \Gbb_h^{-1}=\begin{pmatrix}
1 & 0 & 0\\
0 & 1 & x\\
0 & x & 2h^2+x^2
\end{pmatrix}, $$ and observe that, $ A^t_0{ A_0}+h^2{ A^t}_1{ A_1}=\Gbb_h^{-1}$. Moreover, we have that $$\mathrm{spec} ( A_0^t A_0)=\{1,x^2+1,0\}, \;\mathrm{ and }\;\mathrm{spec} (\Gbb_h^{-1})=\{\lambda_1(h,x),\lambda_2(h,x),\lambda_3(h,x)\},$$ where $\lambda_1(h,x)=1$, $$\lambda_2(h,x) =\dfrac{1}{2}\left(\sqrt{(2h^2+x^2+1)^2-8h^2}+2h^2+x^2+1\right)\sim (x^2+1),$$ and $$\lambda_3(h,x)= \dfrac{1}{2}\left(-\sqrt{(2h^2+x^2+1)^2-8h^2}+2h^2+x^2+1\right)\sim h^2(1+o(h^{-2})).$$  For any $x$, two eigenvalues, $\lambda_1(h,x)$ and $\lambda_2(h,x)$ are of order 0 and one eigenvalue, $\lambda_3(h,x)$, is of order $h^2$. This is expected, as this is the equiregular case, and the growth vector equals to $(2,3)$ is constant everywhere. So, at any $x$, $$\mathrm{det}(\Gbb_h^{-1})=2h^{2}\sim 2h^{2}.$$ In this case, $\varsigma(x)=1$ and $f_h(x)=2\rightarrow f(x)=2$.

The volume form obtained from the approximation scheme is given by $$d\mathcal{P}_o=(1/\sqrt{f_0(x)})|dx\wedge dy\wedge dz|=\dfrac{1}{\sqrt{2}}|dx\wedge dy\wedge dz|.$$
Direct computation using the expression (\ref{ltdisc}), implies that $g^h(X^{0i},X^{0j})=\delta_i^j$ for $i,j=1,2$.
 So, $X^{01},X^{02},X^{11}$ is an adapted frame. Following (\ref{popbari}), as this is an equiregular setting with $r=1$ on $\mathbb R^3$, we have \begin{equation*}
            d\mathcal{P}=\nu_1\wedge\nu_2\wedge\nu_3=\frac{1}{\sqrt{2}}|dx\wedge dy\wedge(dz-xdy)|=\frac{1}{\sqrt{2}}|dx\wedge dy\wedge dz|=d\mathcal{P}_o,
        \end{equation*} where $\nu_1,\nu_2,\nu_3$ are the dual frame to $X^{01},X^{02},X^{11}$.
    \end{example}
    \begin{example}[Martinet Case on $\mathbb{R}^3$]
    On $\mathbb R^3$, consider the smooth vector fields $$X^{01}=\dl_x \text{ and } X^{02}=\dl_y+\frac{x^2}{2}\dl_z.$$ On $\mathbb R^3\setminus\{x=0\}$, $X^{01},X^{02} \textit{ and } X^{11}=[X^{01},X^{02}]=x\dl_z=: -X^{12}$ span $\mathbb R^3$ (step 2). On $\{x=0\}$, $X^{01},X^{02}, X^{11} \textit{ and } X^{21}=[X^{01},X^{11}]=\dl_z=-X^{22}$ span $\mathbb R^3$ (step 3). This is a non-equiregular sub-Riemannian setting. Following definition (\ref{g^inf}), define the sub-Riemannian metric $g^0$ on $\mathbb R^3$ as $$g^0(X)=\inf\{u_1^2+u_2^2\,;\, {\bfu}=(u_i)_{i=1,2}\in\mathbb R^3,\,\sum_{i=1}^2 u_iX^{0i}=X\}.$$ Following definition (\ref{g^s}), define the Riemannian metric $g^h$ on $\mathbb R^3$ as $$g^h(X)=\inf\{u_1^2+u_2^2+h^{-2}u_3^2+h^{-4}u_4^2\,;\, {\bfu}=(u_i)_{i=1,2,3,4}\in\mathbb R^4,\,\sum_{i=1}^2 u_iX^{0i}+u_3X^{11}+u_4X^{12}+u_5X^{21}+u_6X^{22} =X\}.$$ For $X=(a,b,c)\in\mathbb R^3$, computation using a Lagrange multiplier implies that \begin{equation}
        \label{ltdisc2}g^h((a,b,c))=a^2+b^2+\dfrac{(2c-bx^2)^2}{8h^2(h^2+x^2)},
    \end{equation} which is clearly a smoothly defined inner product on $\mathbb R^3$. Moreover, for any $(a,b,c)\in\mathbb R^3$, $$g^0(a,b,c)=a^2+b^2.$$ For $g^0$, which is the limit of $g^h$ as $h\rightarrow 0$, to be finite, $c-\frac{x^2}{2}b$ must be $0$, which means that $X=aX^{01}+bX^{02}$, that is $X$ is horizontal.  

    Denote by $\Gbb$ the representation matrix of $g^h$. We compute $$A_0=\begin{pmatrix}
1 & 0 & 0\\
0 & 1 & \dfrac{x^2}{2}
\end{pmatrix}, A_1=\begin{pmatrix}
0 & 0 & x \\
0 &0  &-x
\end{pmatrix}, A_2=\begin{pmatrix}
0 & 0 & 1\\
0 & 0 & -1
\end{pmatrix},\text{ and }  \Gbb_h^{-1}=\begin{pmatrix}
1 & 0 & 0\\
0 &1 & \dfrac{x^2}{2}\\
0 & \dfrac{x^2}{2} & 2h^4+2h^2x^2+\dfrac{x^4}{4}\end{pmatrix},$$
and observe that $A^t_0{ A_0}+h^2{ A^t}_1{ A_1}+h^4{ A^t}_2{ A_2}=\Gbb_h^{-1}$. Moreover we have that 
 $$\mathrm{spec} ( A_0^t A_0)=\{1,\dfrac{x^2}{4}+1,0\}, \;\mathrm{ and }\;\mathrm{spec} (\Gbb_h^{-1})=\{\lambda_1(h,x),\lambda_2(h,x),\lambda_3(h,x)\},$$ where $\lambda_1(h,x)=1$, $$\lambda_2(h,x)=\frac{1}{8} (8 h^4 + 8 h^2 x^2 +\sqrt{(-8 h^4 - 8 h^2 x^2 - x^4 - 4)^2 - 16 (8 h^4 + 8 h^2 x^2)} + x^4 + 4)\sim \dfrac{x^4}{4}+1,$$ and $$\lambda_3(h,x)= \frac{1}{8}  (8 h^4 + 8 h^2 x^2 -\sqrt{(-8 h^4 - 8 h^2 x^2 - x^4 - 4)^2 - 16 (8 h^4 + 8 h^2 x^2)} + x^4 + 4).$$
 Away from the singular plane $\{x=0\}$, two eigenvalues, $\lambda_1(h,x)$ and $\lambda_2(h,x)$ are of order $0$ and one eigenvalue, $\lambda_3(h,x)$ is of order $h^{2}$ as $\lambda_3(h,x)\sim h^2x^2$. So $$\mathrm{det}(\Gbb_h^{-1})=2h^2x^2+2h^4\sim 2h^2x^2.$$ In this case, $\varsigma(x)=1$ and $f_h(x)=2x^2+2h^2\rightarrow f(x)=2x^2$. 
 
 At $\{x=0\}$, two eigenvalues, $\lambda_1(h,x)$ and $\lambda_2(h,x)$, are of order $0$ and one eigenvalue, $\lambda_3(h,x)$ is of order $h^{4}$, as $\lambda_3(h,x)\sim h^4$. So $$\text{det}(\Gbb_h^{-1})=2h^4\sim 2h^4.$$  This is expected as the growth vector equals $(2,3)$ for $x\neq 0$, and $(2,2,3)$ at $\{x=0\}$.

 The volume form obtained from the approximation scheme is given by $$d\mathcal{P}_o=(1/\sqrt{f_0(x)})|dx\wedge dy\wedge dz|=(1/\sqrt{2}|x|)|dx\wedge dy\wedge dz|.$$ 
 Direct computation using (\ref{ltdisc2}) shows that  $g^h(X^{0i},X^{0j})=\delta_i^j$ for $i,j=1,2$. So, $X^{01},X^{02},X^{11},X^{21}$ is an adapted frame. Following (\ref{popbari}), and as $r=1$ on the equiregular region $\mathbb R^3\setminus\{x=0\}$, we have \begin{equation*}
            d\mathcal{P}=\nu_1\wedge\nu_2\wedge\nu_3=\frac{1}{\sqrt{2}}|dx\wedge dy\wedge(\frac{1}{|x|}dz-\frac{|x|}{2}dy)|=\frac{1}{{\sqrt{2}}|x|}|dx\wedge dy\wedge dz|=d\mathcal{P}_o,
        \end{equation*} where $\nu_1,\nu_2,\nu_3,\nu_4$ are the dual frame to  $ X^{01},X^{02},X^{11},X^{21}.$
\end{example}
In the Heisenberg and the Martinet examples, the initial vector fields are orthonormal and the coefficients pops up because of the dual frame solely. We give an example where this is not the case.
    \begin{example}
        On $\mathbb R^4$, consider the following smooth vector fields: $X^{01}=\dl_x,X^{02}=\dl_y,X^{03}=x\dl_z+y\dl t$. On $\mathbb R^4\setminus\{x=0\}$,  $X^{01},X^{02},X^{03}$ and $X^{11}=[X^{02},X^{03}]=\dl_t:=-X^{12}$ span $\mathbb R^4$. On $\{x=0\}$, $X^{01},X^{02},X^{03}$ and $X^{13}=[X^{01},X^{03}]=\dl_z:=-X^{14}$ span $\mathbb R^4$. Then, this is an equiregular setting with step 1.   Following definition (\ref{g^inf}), define the sub-Riemannian metric $g^0$ on $\mathbb R^4$ as $$g^0(X)=\inf\{u_1^2+u_2^2+u_3^2; {\bfu}=(u_i)_{i=1,2,3}\in\mathbb R^3, u_1X^{01}+u_2X^{02}+u_3X^{03}=X\}.$$ Following definition (\ref{g^s}), define the Riemannian metric $g^h$ on $\mathbb R^3$ as $$g^h(X)=\inf\{u_1^2+u_2^2+u_3^2+h^{-2}\sum_{i=4}^7 u_i; {\bfu}=(u_i)_{i=1,2,3,4,5,6,7}\in\mathbb R^5, \sum_{i=1}^3 u_i X^{0i}+\sum_{i=1}^4 u_{i+3}X^{1i}=X\}.$$ For $X=(a,b,c,d)\in\mathbb R^4$, computations imply that \begin{equation}\label{ltdisc3}
             g^h((a,b,c,d))=a^2+b^2+\dfrac{(cy-dx)^2+2h^2(c^2+d^2)}{2h^2(2h^2+x^2+y^2)},
         \end{equation} which is a smoothly defined inner product on $\mathbb R^3$. Moreover, for any $(a,b,c,d)\in \mathbb R^4$, $$g^0((a,b,c))=a^2+b^2+\dfrac{c^2}{x^2}.$$ For $g^0$, which is the limit of $g^h$ as $h\rightarrow 0$, to be finite, $dx-cy$ must be $0$, which means that $X$ is horizontal.
       
        Denote by $\Gbb_h$ the representation matrix of $g^h$. We compute  $$A_0=\begin{pmatrix}
1 & 0 & 0 &0\\
0 &0& 1 & x\\
0&0&x&y
\end{pmatrix}, A_1=\begin{pmatrix}
0 & 0 & 1&0 \\
0 &0  &-1&0
\\ 0 & 0 & 0&1 \\
0 &0  &0&-1
\end{pmatrix}\text{ and } \Gbb_h^{-1}=\begin{pmatrix}
1 & 0 & 0&1\\
0 & 1 &0& \\
0 & 0 &x^2+2h^2&xy\\
0&0&xy&2h^2+y^2
\end{pmatrix}, $$ and observe that, $ A^t_0{ A_0}+h^2{ A^t}_1{ A_1}=\Gbb_h^{-1}$. Moreover, we have that $$\mathrm{spec} ( A_0^t A_0)=\{1,1,x^2+y^2,0\}, \;\mathrm{ and }\;\mathrm{spec} (\Gbb_h^{-1})=\{\lambda_1(h,x),\lambda_2(h,x),\lambda_3(h,x),\lambda_4(h,x)\},$$ where $\lambda_1(h,x)=\lambda_2(h,x)=1$, $$\lambda_3(h,x)= \dfrac{1}{2} (2 h^2 + \sqrt{4 h^4 - 4 h^2 x^2 + 4 h^2 y^2 + x^4 + 2 x^2 y^2 + y^4} + x^2 + y^2)\sim (x^2+y^2),$$ and $$\lambda_4(h,x)= \dfrac{1}{2}(2 h^2 - \sqrt{4 h^4 - 4 h^2 x^2 + 4 h^2 y^2 + x^4 + 2 x^2 y^2 + y^4} + x^2 + y^2)\sim h^2.$$  For any $x\neq0$, three eigenvalues, $\lambda_1(h,x),\lambda_2(h,x)$ and $\lambda_3(h,x)$ are of order 0 and one eigenvalue, $\lambda_4(h,x)$, is of order $h^2$. This is expected, as this is the equiregular case, and the growth vector equals to $(3,4)$ is constant everywhere. So, at any $x\neq0$, $$\mathrm{det}(\Gbb_h^{-1})=2h^{2}x^2\sim 2h^{2}x^2.$$ In this case, $\varsigma(x)=1$ and $f_h(x)=2x^2\rightarrow f(x)=2x^2$.
 Then, the volume given by the approximation is given by $$d\mathcal{P}_o=(1/\sqrt{2(x^2+y^2)})|dx\wedge dy\wedge dz\wedge dt|.$$
 Now, direct computation using (\ref{ltdisc3}) shows that $$g^h(X^{03})=\frac{x^2+y^2}{1+\frac{x^2+y^2}{h}}=:a(x,y)^2.$$ Now, setting $$\{{X}_1,{X}_2,{X}_3,{X}_4\}=\{X^{01},X^{02},(1/a(x,y))X^{03},X^{11}\},$$ we get that $g^h({X}_i,{X}_j)=\delta_i^j$ for $i,j=1,...,4$. Thus $\{{X}_1,{X}_2,{X}_3,{X}_4\}$ is an adapted frame. Then following (\ref{bi's}), and as $$[X^{01},X^{03}]=\dl_z=\dfrac{1}{xa(x,y)}{X}_3+\dfrac{-y}{x}{X}_4,$$ we get that $b_{23}^4=1$ and $b_{13}^4=(y/x)$. Then, following (\ref{B_j's}), we get that $$B_2=\left(\frac{2}{x^2}(x^2+y^2)\right).$$ Thus, (\ref{popbari}) implies that $$d\mathcal{P}=\frac{|x|}{\sqrt{2}\sqrt{x^2+y^2}}\nu_1\wedge \nu_2\wedge\nu_3\wedge\nu_4=\frac{|x|}{\sqrt{2}\sqrt{x^2+y^2}}|dx\wedge dy\wedge\frac{1}{|x|}dz\wedge dt|=d\mathcal{P}_o,$$ where $\nu_1,\nu_2,\nu_3,\nu_4$ are the dual frame to 
$ X^{01},X^{02},X^{03},X^{11}.$\end{example}
\section{Convergence Of Spectrum}
\label{kitrez}
In this section, we assume that $M$ is compact. We define Laplace type operators: $\Delta_0$ associated with the sub-Riemannian structure
and $\Delta_h$ with the approximating sequence of Riemannian metrics. These operators have compact resolvent and hence a spectrum consisting
only of eigenvalues. The question we address is whether the spectrum of the approximating operator converges to the spectrum of the limiting one
in the following sense:
\[
  \forall k\geq 0,~~\lambda_k(\Delta_h) \tendsto{h}{0} \lambda_k(\Delta_0).
\]

We will consider two cases :
\begin{enumerate}
\item The operators $\Delta_0$ and $\Delta_h$ are defined relatively to a ($h$-independent) smooth volume form $\omega$.
\item The operators $\Delta_h$ are the Riemannian Laplace operators associated with the metric $g_h$ and $\Delta_0$ is
  associated with Popp's volume. 
\end{enumerate}

In the latter case, we will give the answer when the sub-Riemannian structure is equiregular.
In the former case, this distinction does not play any role.

We start by some definitions.
\subsection{Definitions And Notations}\label{mm328l}
Let $\omega$ be a smooth volume form: in any local coordinates $x$, $\omega$ can be written as $\phi(x)dx$ with $\phi$ smooth. 

We recall the definition of divergence relatively to $\omega$. 
\begin{definition}
         Fix a smooth volume form $\omega$ on $M$ and let $X$ be a smooth vector field.
         The divergence $\mathrm{div}_{\omega}(X)$ is defined as the function satisfying:
         For any $u,v\in\mathcal{C}^\mathcal{1}(M)$,
         \begin{equation}
             \label{definitionofdiv}\int_M uXvd\omega=\int_M \left[\left(-X-\mathrm{div}_{\omega}(X)\right)u\right]vd\omega.
         \end{equation}
     \end{definition}
     In the sequel we will use this definition for $\omega$ and for the sequence $\omega_h=\mathrm{dvol}g^h$. We will denote by $*_\omega$ and $*_h$ the adjoint, and by $\mathrm{div}_\omega$ and $\mathrm{div}_h$, the divergence with respect to these volume forms respectively.

    We denote by $L^2_{\omega}(M):=L^2_{\omega}((M,\mathbb R),d\omega)$ the Hilbert space associated to $d\omega$, given by 
$$L^2_{\omega}(M):=\left\{u:M\rightarrow\mathbb{R};\;\|u\|_{L^2_{\omega}(M)}^2:=\int_M |u|^2d\omega <\mathcal{1}\right\}.$$

     We conclude this paragraph by recalling the definition of $H_\omega^s(M):$ the Sobolev space with respect to $\omega.$

For $s\in \mathbb R$, define $\Lambda^s=(\text{Id}+\Delta_1)^{\frac{s}{2}},$ where Id is the identity operator.
Define the space $H_{\omega}^s(M)$ with respect to $\omega$ as
\begin{equation}
    \label{H^s}H^s_\omega(M)=\{u\in L^2_{\omega}(M);\; \norm{u}_{H^s_\omega(M)}:=\norm{\Lambda^su}_{L^2_{\omega}(M)}< \mathcal{1}\}.
\end{equation} 
Denote by $\mathcal{L}$ the space of linear bounded functions from $H^2_{\omega}(M)$ to $L^2_{\omega}(M)$.
\subsubsection{SubLaplace Operator}
On $L^2_{\omega}(M)$, we define the subLaplace operator as \begin{equation}
    \label{sublap} \Delta_0=\sum_{j=0}^{N_0} (X^{0j})^{*_\omega}X^{0j}=\sum_{j=0}^{N_0} \left(-(X^{0j})^2-\text{div}_{\omega}(X^{0j})X^{0j}\right),
  \end{equation} where the star denote the adjoint with respect to $d\omega$.

This operator $\Delta_0$ is essentially self-adjoint on $C^\infty(M)$ and sub-elliptic. It follows that it has a discrete spectrum.
  
\subsubsection{Approximating sequence}
We now proceed to define the approximating sequence of operators in both cases, when the volume form is fixed, and when it is associated to $g^h.$
\paragraph{With a fixed volume form}\hfill 

For any $h>0$, we define on $L^2_{\omega}(M)$, the family of elliptic operators:
$$ \Delta_h=\sum_{i=0}^{\bfr}\sum_{j=1}^{N_i}h^{2i}(X^{ij})^{*_\omega}X^{ij}
=\sum_{i=0}^{\bfr}\sum_{j=1}^{N_i}h^{2i}\left( -(X^{ij})^{2}-\text{div}_{\omega}(X^{ij})X^{ij}\right),$$
where the star denotes the adjoint with respect to $d\omega$.

This operator is essentially self-adjoint on $C^\infty(M)$ and elliptic with compact resolvent.
We observe that $\Delta_h$ is not the Riemannian Laplace operator associated with $g_h$.

\paragraph{Riemannian approximation}\hfill

We recall that the Hausdorff dimension of the equiregular region is assumed to be constant and denoted by $\bfQ$.
We also denote by  $L^2_h(M)$ the Hilbert space associated to $h^{\bfQ-d}\text{dvol}g^h$, given by
$$L^2_h(M):=\left\{u:M\rightarrow\mathbb{R};\;\norm{u}_{L^2_h(M)}^2:=\int_M |u|^2h^{\bfQ-d}\text{dvol}g^h <\mathcal{1}\right\}.$$
We recall that $*_h$ and $\text{div}_h$ denote the adjoint and divergence operations in this Hilbert space.

For any $h>0$, we define on $L^2_h(M)$, the family of elliptic operators:
$$ \til\Delta_h=\sum_{i=0}^{\bfr}\sum_{j=1}^{N_i}h^{2i}(X^{ij})^{*_h}X^{ij}=\sum_{i=0}^{\bfr}\sum_{j=1}^{N_i}h^{2i}\left( -(X^{ij})^{2}-\text{div}_h(X^{ij})X^{ij}\right).$$
This operator $\til\Delta_h$  is (up to the renormalizing factor $h^{\bfQ-d}$) the Riemannian Laplace operator associated with $g^h$, as it can be seen from the expression for the metric in a coordinate frame. 
\subsection{Convergence Of Spectrum With Fixed Volume Form}
In this section, we are interested in proving the convergence of the spectrum in the fixed volume form case.
The strategy is as follows: consider a sequence of eigenpairs $(\lambda_h,u_h)$ such that $\lambda_h$ is bounded. We want to extract a subsequence of $u_h$ that converges in $H^s_\omega(M)$ for large $s$. In order to do that, we need to gain compactness in $H^{s'}_\omega$ for $s'>s$. The key point is to prove a uniform bound on the eigenfunctions in the latter spaces. One way to do that is by proving local subelliptic estimates that are uniform with respect to $h$ (see \cite{helffer20052,kohn2005hypoellipticity,colindeverdiere:hal-02535865}). 

Here, we state a uniform parameter-dependent version of the subelliptic estimate and we give the proof in the appendix.

\begin{theorem}[Uniform Subelliptic Estimate]\label{sbellest}
       There exists $\epsilon>0$, such that, $\forall s\in \mathbb R,~\exists C(s)>0,~\forall h\in[0,1],~ \forall u\in\mathcal{C}^\mathcal{1}(M),$
      \begin{equation}\label{SEE}           
       \norm{u}_{H_{\omega}^{\epsilon+s}(M)}\leq C(s)\left(\norm{\Delta_hu}_{H_{\omega}^{s}(M)}+\norm{u}_{H_{\omega}^{s}(M)}\right).
      \end{equation}

   \end{theorem}
   \begin{corollary}\label{correas}
  There exists $\epsilon>0$ such that for any $n\in\mathbb{N}$, there exists $C(n)>0$, for any $u$ smooth on $M$, and any $h\in[0,1]$, we have,
  \begin{equation}\label{ft5}
     \norm{u}_{H_{\omega}^{n\epsilon}(M)}\leq C(n)\left(\norm{\Delta_h^nu}_{L^2_{\omega}(M)}+\norm{\Delta_h^{n-1}u}_{L^2_{\omega}(M)}+...\norm{\Delta_hu}_{L^2_{\omega}(M)}+\norm{u}_{L^2_{\omega}(M)}\right).
 \end{equation}
 \end{corollary}
 \begin{proof}
   Take $s=\epsilon$ in (\ref{SEE}), we get
   $$\norm{u}_{H_{\omega}^{2\epsilon}(M)}\leq C\left(\norm{\Delta_hu}_{H_{\omega}^{\epsilon}(M)}+\norm{u}_{H_{\omega}^{\epsilon}(M)}\right)\leq C\left(\norm{\Delta_hu}_{H_{\omega}^{\epsilon}(M)}+\norm{\Delta_hu}_{L^2_{\omega}(M)}+\norm{u}_{L^2_{\omega}(M)}\right).$$ Now, apply (\ref{SEE}) for s=0 and $\Delta_hu$, to get that $$\norm{\Delta_hu}_{H_{\omega}^{\epsilon}(M)}\leq C\left(\norm{\Delta_h^2u}_{L^2_{\omega}(M)}+\norm{\Delta_hu}_{L^2_{\omega}(M)}\right).$$ We get (\ref{ft5}) for $n=2$. We  go on recursively to get (\ref{ft5}) for any $n$. 
 \end{proof}
 This kind of uniform estimates have also been studied in \cite{colindeverdiere:hal-02535865}. In the latter reference, the authors state as a remark that it should imply the spectral convergence; we give a full proof here.
 
 We need the following lemmas.
          \begin{theorem}[Convergence Of Spectrum]\label{convofbddspec}
      Let $(h_n)_{n\geq0}$ be a sequence that goes to 0 and $(u_n)_{n\geq0}$ be a sequence of normalized eigenfunctions of $\Delta_{h_n}$. Let $(\mu_n)_{n\geq0}$ be the associated sequence of eigenvalues. Assume that the sequence $(\mu_n)_{n\geq0}$ is bounded. Then, there exists an eigenpair $(\lambda_0,v_0)$ of $\Delta_0$ such that up to extracting a subsequence, $u_{n_k}$ converge to $v_0$ strongly in $H^s_\omega(M)$ for any $s$, and $\mu_{n_k}$ converges to $\lambda_0$.
    
      \end{theorem}
      \begin{proof}
    Since the sequence $(\mu_n)_{n\geq0}$ is bounded, it has a subsequence $(\mu_{n_k})_{k\geq0}$ that converges to some $\lambda_0$. 
             
             Now, $(u_{n_k})_{k\geq0}$ is a sequence of smooth functions as they are eigenfunctions of elliptic Riemannian Laplace operators. If we apply the uniform estimate  (\ref{ft5}) to $u_{n_k}$ we get that for any $s$, \begin{equation}
              \label{efunbdd}\norm{u_{n_k}}_{H_{\omega}^s(M)}\leq C\norm{(\Delta_h+1)^{(\bfr+1)s}u_{n_k}}_{L^2_{\omega}(M)}=(1+\mu_{n_k}^{(\bfr+1)s})\norm{u_{n_k}}_{L^2_{\omega}(M)}=(1+\mu_{n_k}^{(\bfr+1)s}).
          \end{equation}
          Since $(\mu_n)_{n\geq0}$ is bounded, it implies that the sequence $(u_{n_k})_{k\geq0}$ is bounded in $H_{\omega}^s(M)$ for any $s$. Using that $H^{s+1}_\omega(M)$ embeds compactly into $H^s_\omega(M)$, and a diagonal argument, we can extract a subsequence $u_{n_k}$ that converges strongly in any $H^s_\omega(M)$ to $v_0$.

Since 

$$\Delta_h=\Delta_0+\sum_{i=1}^{\bfr}\sum_{j=1}^{N_i}h^{2i}(X^{ij})^{*_\omega}X^{ij},$$

we can pass to the limit in the eigenvalue equation to obtain that 

$$\Delta_0 v_0=\lambda_0v_0.$$

          Since $u_n$ is normalized and converges to $v_0$ strongly in $L^2_\omega(M)$, the latter function cannot be zero. Thus, $(\lambda_0,v_0)$ is an eigenpair of $\Delta_0$.

     \end{proof}
      We deduce the following corollary using min-max theorem.
      \begin{corollary}
        \label{convoforderedspec} Denote by $(\lambda_k)_{k\geq0}$ and $(\lambda_k(h))_{k\geq0}$ the ordered spectrum of $\Delta_0$ and $\Delta_h$ respectively, counted with multiplicities. Then, for any $k\geq0$ fixed, we have
        \begin{equation}
        \label{balaminmax1} \lim_{h\rightarrow 0}\lambda_k(h)=\lambda_k.
    \end{equation} 
    \end{corollary}
    \begin{proof}
    It is enough to prove that any converging subsequence of $\lambda_k(h)$ converges to $\lambda_k.$ We still denote by $\lambda_k(h)$ such a sequence and by $\til\lambda$ its limit.
    
    We first prove that $\til\lambda\leq \lambda_k.$ 

    Denote by $V_0$ the space spanned by a set of eigenfunctions that correspond to the first $k$ eigenvalues of $\Delta_0$. Using min-max theorem, we compute
    
    \begin{equation*}
    \begin{split}
           \lambda_k(h)&\leq \Lambda_h(V_0):=\max_{u\in V_0}\left\{\dfrac{\lan\Delta_hu,u\ran_{L^2_\omega(M)}}{\norm{u}_{L^2_\omega(M)}^2}\right\}\\&=\max_{u\in V_0}\left\{\dfrac{\lan \Delta_0u,u\ran_{L^2_\omega(M)}}{\norm{u}_{L^2_\omega(M)}^2}+\dfrac{\sum_{i=1}^{\bfr}\sum_{j=1}^{N_i}h^{2i}(X^{ij})^*X^{ij}u}{\norm{u}_{L^2_\omega(M)}^2}\right\} \\&\leq\max_{u\in V_0}\left\{\dfrac{\lan \Delta_0u,u\ran_{L^2_\omega(M)}}{\norm{u}_{L^2_\omega(M)}^2}\right\}+\max_{u\in V_0}\left\{\dfrac{\sum_{i=1}^{\bfr}\sum_{j=1}^{N_i}h^{2i}(X^{ij})^*X^{ij}}{\norm{u}_{L^2_\omega(M)}^2}\right\}\underset{h\rightarrow0}{\rightarrow} \Lambda_0(V_0)=\lambda_k.
    \end{split}
    \end{equation*}
    Indeed, since $V_0$ is a finite dimensional space of smooth functions, we have 

    $$\max_{u\in V_0}\left\{\dfrac{\sum_{i=1}^{\bfr}\sum_{j=1}^{N_i}h^{2i}(X^{ij})^*X^{ij}}{\norm{u}_{L^2_\omega(M)}^2}\right\}=O(h^2).$$
    
    We get that  $\til\lambda\leq \lambda_k.$

    Now we prove the opposite inequality.

    Let $V_h$ denote the space spanned by a set of orthonormal eigenfunctions that correspond to the first $k$ eigenvalues of $\Delta_h$, $V_h=\mathrm{span}\{v_1(h),...,v_k(h)\}$. Using theorem \ref{convofbddspec}, we build a subsequence such that each $v_i(h)$ converges to $v_i$ in $H^s_\omega(M)$ for any $s$. By orthonormality of $v_i(h),$ $v_i$ are also orthonormal. We define the space of dimension $k, V_0=\mathrm{span}\{v_1,...,v_k\},$ and we observe that by orthonormality it is of dimension $k$. For any $i$, we have 
    
    $$\dfrac{\lan\Delta_0 v_i,v_i\ran_{L^2_\omega(M)}}{\norm{v_i}_{L^2_\omega(M)}^2}=\lim_{h\rightarrow0}\dfrac{\lan\Delta_h v_i(h),v_i(h)\ran_{L^2_\omega(M)}}{\norm{v_i(h)}_{L^2_\omega(M)}^2}\leq\til\lambda.$$

Taking the maximum over $V_0$ and applying the min-max theorem, we get that 

$$\lambda_k\leq \til\lambda.$$

We conclude.

    \end{proof}
 \subsection{Convergence of Spectrum In Equiregular Case}
 Suppose now the $\omega$ is the Popp's volume. Recall from the previous section, that $\lim_{h\rightarrow0} h^{{\bfQ}-d}\mathrm{dvol}g^h=d\omega.$ Denote by $\Psi_h$ the quotient between the two volume forms; 

 $$\Psi_h=\dfrac{h^{{\bfQ}-d}\mathrm{dvol}g^h}{dw}.$$

 When the sub-Riemannian setting is equiregular, the volume $\omega$ is smooth, ${\bfQ}$ is constant on $M$ and $\Psi_h$ is smooth and positive on $M$. Moreover, in the preceding section we proved that $\Psi_h$ (and thus $1/\Psi_h$) converges uniformly to the constant function $\mathbbm 1$ on $M$.

 Denote by $\til\lambda_k(h)$ the k-th eigenvalue of $\til\Delta_h.$ We aim at proving that $\lim_{h\rightarrow0}\til\lambda_k(h)=\lambda_k,$ where we recall that $\lambda_k$ denotes the k-th eigenvalue of $\Delta_0.$ By corollary \ref{convoforderedspec}, we know that $\lambda_k(h)\rightarrow \lambda_k$ as $h\rightarrow0.$ So, it is enough to prove that $\til\lambda_k(h)\sim\lambda_k(h)$ as $h\rightarrow0.$ We will do this using min-max theorem.

\begin{theorem}
     Denote by $(\lambda_k)_{k\geq0}$ and $(\til\lambda_k(h))_{k\geq0}$ the ordered spectrum of $\Delta_0$ and $\til\Delta_h$ respectively, counted with multiplicities. Then, for any $k\geq0$ fixed, we have \begin{equation}
        \label{balaminmax} \lim_{h\rightarrow 0}\til\lambda_k(h)=\lambda_k.
    \end{equation} 
\end{theorem}
 \begin{proof}
 First, for any $u$,  notice that 

 $$\left(1+\norm{\dfrac{1}{\Psi_h}-1}_{\infty}\right)^{-1}\norm{u}_{L^2_\omega(M)}\leq \norm{u}_{L^2_h(M)}\leq(1+\norm{{\Psi_h}-1}_{\infty})\norm{u}_{L^2_\omega(M)}.$$ 

 Let $V_h$ be the space spanned by a set of orthonormal eigenfunctions that correspond to the first $k$ eigenvalues of $\Delta_h.$ We compute 
 \begin{equation*}
     \begin{split}
         \til\lambda_k(h)&\leq \til\Lambda(V_h):=\max_{u\in V_h}\left\{\dfrac{\sum_{i,j}h^{2i}\norm{X^{ij}u}^2_{L^2_h(M)}}{\norm{u}^2_{L^2_h(M)}}\right\}\\&\leq \dfrac{(1+\norm{{\Psi_h}-1}_{\infty})}{\left(1+\norm{\dfrac{1}{\Psi_h}-1}_{\infty}\right)^{-1}}\max_{u\in V_h}\left\{\dfrac{\sum_{i,j}h^{2i}\norm{X^{ij}u}^2_{L^2_\omega(M)}}{\norm{u}^2_{L^2_\omega(M)}}\right\}\\&=a(h)\lambda_k(h),
     \end{split}
 \end{equation*}
with 

$$a(h)=\dfrac{(1+\norm{{\Psi_h}-1}_{\infty})}{\left(1+\norm{\dfrac{1}{\Psi_h}-1}_{\infty}\right)^{-1}}\underset{h\rightarrow0}{\longrightarrow1}.$$

Now, Let $\til V_h$ be the space spanned by a set of orthonormal eigenfunctions that correspond to the first $k$ eigenvalues of $\til\Delta_h.$ We compute 
 \begin{equation*}
     \begin{split}
         \lambda_k(h)&\leq \til\Lambda(\til V_h):=\max_{u\in\til V_h}\left\{\dfrac{\sum_{i,j}h^{2i}\norm{X^{ij}u}^2_{L^2_\omega(M)}}{\norm{u}^2_{L^2_\omega(M)}}\right\}\\&\leq \dfrac{\left(1+\norm{\dfrac{1}{\Psi_h}-1}_{\infty}\right)}{(1+\norm{{\Psi_h}-1}_{\infty})^{-1}}\max_{u\in\til V_h}\left\{\dfrac{\sum_{i,j}h^{2i}\norm{X^{ij}u}^2_{L^2_h(M)}}{\norm{u}^2_{L^2_h(M)}}\right\}
         \\&=a(h)\til\lambda_k(h),
     \end{split}
 \end{equation*}

We finally have 

$$(a(h))^{-1}\lambda_k(h)\leq \til\lambda_k(h)\leq a(h)\lambda_k(h)$$

This proves that $\lambda_k(h)\sim\til\lambda_k(h).$ 
\end{proof}

       In the non-equiregular setting, the singular set $\mathcal{Z}$ is non-empty, and the volume form associated to the Riemannian approximation blows up near $\mathcal{Z}$. This fact implies that several statements that we used in the proof may not be valid anymore. For instance, $M\setminus\mathcal{Z}$ is not compact anymore, and it is not clear whether the subelliptic estimate remains true. As a consequence, if we denote by  $\mathcal{C}^\infty_0(M\setminus\mathcal{Z})$ the set of all smooth functions on $ M\setminus\mathcal{Z}$ that vanishes on a neighborhood of $\mathcal{Z}$, we don't know in full generality whether $(\Delta_0,\mathcal{C}^\infty_0(M\setminus\mathcal{Z}))$ is essentially self-adjoint. This question has been studied in \cite{boscain2013laplace}\cite{boscain2016self}\cite{franceschi2020essential}\cite{prandi2018quantum} where conditions on the singular set $\mathcal{Z}$ are given that imply the essential selfadjointness of $(\Delta_0,\mathcal{C}^\infty_0(M\setminus\mathcal{Z}))$. For instance, the authors in \cite{franceschi2020essential} prove that this holds for the sublaplacian defined with respect to Popp's volume,  under some regularity assumption on the sub-Riemannian structure (Popp regularity). 

This observation raises new questions in our settings: do we have spectral convergence when the limiting operator is essentially selfadjoint? If it is not and we still have spectral convergence, what is the selfadjoint extension in the limit?
   \begin{appendix}
    \section{\textbf{Appendix}: Uniform Subelliptic Estimate For Fixed Volume Form}
    We give in this appendix the proof of theorem \ref{sbellest}. We will follow \cite{helffer20052} and adapt Kohn's proof
(see also \cite{colindeverdiere:hal-02535865}\cite{kohn2005hypoellipticity}). We note that we do not aim at having an optimal gain of regularity
(the optimal order of subellipticity is $1/\bfr$ as proved in \cite{rothschild1976hypoelliptic}).

We recall that any zero order pseudo-differential operator on $M$ is bounded in any Sobolev space, and that any negative order pseudo-differential operator is bounded in $L^2_\omega(M)$. Moreover, the difference of two pseudo-differential operators with the same principle symbol of order $p$ is of order $p-1$.
Finally, the norm of a pseudo-differential operator in $\mathcal{L}(L^2_\omega(M))$ depends only on a finite number of derivatives of its principle symbol.
See \cite{calderon1972class}\cite{zworski2022semiclassical}.

 Recall that for $s\in \mathbb R$, define $\Lambda^s=(\text{Id}+\Delta_1)^{\frac{s}{2}},$ where Id is the identity operator. In this section, we will be dealing with the fixed volume form so we simply write $*$ instead of $*_\omega$ for the adjoint with respect to this volume form.
 
 We first prove (\ref{SEE}) for $s=0$.
 \begin{theorem}\label{subellestfs=0}
    There exists $\epsilon>0$ and a constant $C>0$ such that for all $h\in [0,1]$ and all $u\in\mathcal{C}^\mathcal{1}(M)$, 
   \begin{equation}\label{SEEprev}
       \norm{u}_{H_{\omega}^{\epsilon}(M)}\leq C\left(\norm{\Delta_hu}_{L^2_{\omega}(M)}+\norm{u}_{L^2_{\omega}(M)}\right).
   \end{equation} 
   \end{theorem}
   \begin{proof}
     The proof is a direct adaptation of Kohn's proof, so we just state the steps of the proof. Denote by $\textbf{P}$ the set of all pseudo-differential operators of order zero. For all $\epsilon\in ]0,1]$, we define 
     $$\textbf{P}_\epsilon=\left\{P\in \textbf{P}; \norm{Pu}_{H_{\omega}^\epsilon(M)}\leq C(\epsilon)\left(\norm{\Delta_hu}_{L^2_{\omega}(M)}+\norm{u}_{L^2_{\omega}(M)}\right), \forall u\in C^\mathcal{1}(M),\forall h\in[0,1]\right\}.$$ 

    Using standard pseudodifferential calculus, we obtain that
     \begin{itemize}
    \item For all $\epsilon_1\leq \epsilon_2$, $\textbf{P}_{\epsilon_2}\subseteq\textbf{P}_{\epsilon_1}$. 
    \item If $P\in \textbf{P}_\epsilon$ then $P^*\in \textbf{P}_\epsilon$ for any $\epsilon\leq\frac{1}{2}$.
    \item $\textbf{P}_{\epsilon}$ is left and right ideal in $\textbf{P}$.
    \item For all $j=1,...,n_0$ and all $\epsilon\leq \frac{1}{2}$, $X^{0,j}\Lambda^{-1}\in\textbf{P}_\epsilon$. 
    \item If $P\in\textbf{P}_\epsilon$, then $[X^{0,j},P]\in \textbf{P}_\frac{\epsilon}{2}$ for all $j=1,...,n_0$ and $\epsilon\leq\frac{1}{2}$.
    \item For $0\leq i\leq {\bfr}$, for all $1\leq j\leq N_i$, we have $X^{ij}\Lambda^{-1}\in\textbf{P}_{\frac{1}{2^{i+1}}}$ .
    \item  We have $\textbf{P}_\frac{1}{2^{{\bfr}+1}}=\textbf{P}$.
       \end{itemize}
       We deduce (\ref{SEEprev}).
     \end{proof}
     Before proceeding to the proof of \ref{sbellest}, we record a useful result.
       \begin{proposition}\label{akhhh}
         The following holds true: $\forall0\leq i\leq {\bfr},\forall 1\leq j\leq N_i,\forall \z>0,\forall h>0,\forall u\in\mathcal{C}^\mathcal{1}(M),$
         \begin{equation}
              \label{e5rmmhm?}\norm{h^iX^{ij}u}_{L^2_\omega(M)}\leq \frac{\z}{\sqrt{2}}\norm{\Delta_hu}_{L^2_\omega(M)}+\frac{1}{\sqrt{2}\z}\norm{u}_{L^2_\omega(M)}.
          \end{equation} 
          \begin{proof}
             We compute \begin{equation*}
                 \begin{split}
                \norm{h^iX^{ij}u}_{L^2_\omega(M)}^2&\leq\sum_{i=0}^{\bfr}\sum_{j=1}^{N_i}\norm{h^iX^{ij}u}_{L^2_\omega(M)}^2=\lan h^iX^{ij}u,h^iX^{ij}u\ran_{L^2_\omega(M)}\\&=\left\lan \sum_{i=0}^{\bfr}\sum_{j=1}^{N_i}h^{2i}(X^{ij})^*X^{ij}u,h^iu\right\ran_{L^2_\omega(M)}=\lan\Delta_hu,h^iu\ran_{L^2_\omega(M)}\\&\leq \norm{\Delta_hu}_{L^2_\omega(M)}\norm{u}_{L^2_\omega(M)}\leq \frac{\z^2}{2}\norm{\Delta_hu}_{L^2_\omega(M)}^2+\frac{1}{2\z^2}\norm{u}_{L^2_\omega(M)}^2.
                 \end{split}
             \end{equation*}
          \end{proof}
          \end{proposition}
         We will restate and prove theorem \ref{sbellest}.
         \begin{theorem}
       There exists $\epsilon>0$, such that, $\forall s\in \mathbb R,~\exists C(s)>0,~\forall h\in[0,1],~ \forall u\in\mathcal{C}^\mathcal{1}(M),$
      \begin{equation}        
       \norm{u}_{H_{\omega}^{\epsilon+s}(M)}\leq C(s)\left(\norm{\Delta_hu}_{H_{\omega}^{s}(M)}+\norm{u}_{H_{\omega}^{s}(M)}\right).
      \end{equation} \end{theorem}      
      \begin{proof}
      Let $s\in\mathbb{R}$ and apply (\ref{SEEprev}) to $\Lambda^su$: \begin{equation}\label{a5}
    \norm{\Lambda^su}_{H_{\omega}^\epsilon(M)}\leq c\big(\norm{\Delta_h\Lambda^su}_{L^2_{\omega}(M)}+\norm{\Lambda^su}_{L^2_{\omega}(M)}\big).
\end{equation} Observe that \begin{equation}\label{Bb}
\norm{\Delta_h\Lambda^su}_{L^2_{\omega}(M)}\leq\norm{[\Delta_h,\Lambda^s]u}_{L^2_{\omega}(M)}+\norm{\Delta_hu}_{H_{\omega}^s(M)}.
\end{equation} We compute \begin{equation}\label{t'}
\begin{split}
      \norm{[\Delta_h,\Lambda^s]u}_{L^2_{\omega}(M)}&\leq \sum_{i=0}^{\bfr}\sum_{j=1}^{N_i}h^{2i} \norm{[(X^{ij})^*X^{i,j},\Lambda^s]u}_{L^2_{\omega}(M)}\\&\leq \sum_{i=0}^{\bfr}\sum_{j=1}^{N_i}h^{2i}\norm{[(X^{ij})^*,\Lambda^s]{X^{i,j}}u}_{L^2_{\omega}(M)}+\sum_{i=0}^{\bfr}\sum_{j=1}^{N_i}h^{2i}\norm{(X^{ij})^*[{X^{i,j}},\Lambda^s]u}_{L^2_{\omega}(M)}\\&= I_s+J_s.
\end{split}
\end{equation}
Combining (\ref{Bb}) and (\ref{t'}), we have \begin{equation}
    \label{t5ntaktrro} \norm{\Delta_h\Lambda^su}_{L^2_{\omega}(M)}\leq I_s+J_s+\norm{\Delta_hu}_{H_{\omega}^s(M)}.
\end{equation}
We first deal with $I_s$. Fix $s\in\mathbb R$. Using that $[(X^{ij})^*,\Lambda^s]\Lambda^{-s}$ and $[X^{i,j},\Lambda^s]\Lambda^{-s}$ are bounded on $L^2_\omega(M)$,
we compute \begin{equation}
    \label{t} \begin{split}
        I_s&=\sum_{i=0}^{\bfr}\sum_{j=1}^{N_i}h^{2i}\norm{[(X^{ij})^*,\Lambda^s]{X^{i,j}}u}_{L^2_{\omega}(M)}\\&=\sum_{i=0}^{\bfr}\sum_{j=1}^{N_i}h^{2i}\norm{[(X^{ij})^*,\Lambda^s]\Lambda^{-s}\Lambda^s{X^{i,j}}u}_{L^2_{\omega}(M)}\\&\leq c(s)\sum_{i=0}^{\bfr}\sum_{j=1}^{N_i}h^{2i}\norm{\Lambda^s{X^{i,j}}u}_{L^2_{\omega}(M)}\\&\leq c(s)\sum_{i=0}^{\bfr}\sum_{j=1}^{N_i}h^{2i}\left(\norm{X^{i,j}\Lambda^su}_{L^2_{\omega}(M)}+\norm{[X^{i,j},\Lambda^s]u}_{L^2_{\omega}(M)}\right)\\&\leq c(s)\sum_{i=0}^{\bfr}\sum_{j=1}^{N_i}h^{2i}\left(\norm{X^{i,j}\Lambda^su}_{L^2_{\omega}(M)}+C(s)\norm{u}_{H_{\omega}^s(M)}\right).
    \end{split}
\end{equation}
Using proposition \ref{akhhh}, $\exists c_1(s)\forall \z,\exists C_1(s,\z),\forall h\leq1,\forall u\in\mathcal{C}^\mathcal{1}(M),$ \begin{equation}
    \label{akhhh2} I_s\leq c_1(s)\z\norm{\Delta_h\Lambda^su}_{L^2_\omega(M)}+C_1(s,\z)\norm{u}_{L^2_\omega(M)}.
\end{equation}
We now deal with $J_s$. Observe that for any $h\leq 1$, \begin{equation}
\label{fantarni1}
    \begin{split}
        J_s&=\sum_{i=0}^{\bfr}\sum_{j=1}^{N_i}h^{2i}\norm{(X^{ij})^*[{X^{i,j}},\Lambda^s]u}_{L^2_{\omega}(M)}\\&\leq \sum_{i=0}^{\bfr}\sum_{j=1}^{N_i}h^{2i}\norm{[{X^{i,j}},\Lambda^s](X^{ij})^*u}_{L^2_{\omega}(M)}+\sum_{i=0}^{\bfr}\sum_{j=1}^{N_i}h^{2i}\norm{[(X^{ij})^*,[{X^{i,j}},\Lambda^s]]u}_{L^2_{\omega}(M)}\\&\leq\sum_{i=0}^{\bfr}\sum_{j=1}^{N_i}h^{2i}\norm{[{X^{i,j}},\Lambda^s](X^{ij})^*u}_{L^2_{\omega}(M)}+c(s)\norm{u}_{H^s_\omega(M)}\\&= K_s+c(s)\norm{u}_{H^s_\omega(M)},
    \end{split}
\end{equation} 
where the second inequality is due to the fact that $[(X^{ij})^*,[{X^{i,j}},\Lambda^s]]\Lambda^{-s}$ is of order 0.\\
Following the steps as in (\ref{t}), we have that
\begin{equation}
   \begin{split}
        K_s&\leq c(s)\sum_{i=0}^{\bfr}\sum_{j=1}^{N_i}h^{2i}\left(\norm{(X^{i,j})^*\Lambda^su}_{L^2_{\omega}(M)}+C(s)\norm{u}_{H_{\omega}^s(M)}\right)\\&\leq  c(s)\sum_{i=0}^{\bfr}\sum_{j=1}^{N_i}h^{2i}\left(\norm{X^{i,j}\Lambda^su}_{L^2_{\omega}(M)}+\norm{c^{ij}\Lambda^su}_{L^2_{\omega}(M)}+C(s)\norm{u}_{H_{\omega}^s(M)}\right),
    \end{split} 
\end{equation} 
where $c^{ij}=\text{div}_\omega(X^{ij})$ are smooth functions on $M$. Thus, we get that
\begin{equation}
    \label{fantarni2} K_s\leq c(s)\sum_{i=0}^{\bfr}\sum_{j=1}^{N_i}h^{2i}\left(\norm{X^{i,j}\Lambda^su}_{L^2_{\omega}(M)}+\sup_{m\in M}|c^{ij}(m)|C(s)\norm{u}_{H_{\omega}^s(M)}\right).
\end{equation}
Then, again using proposition \ref{akhhh}, we get that
$$\exists c_2(s),\forall \z,\exists C_2(s,\z),\forall h\leq 1,\forall u\in\mathcal{C}^\mathcal{1}(M),$$
\begin{equation}
    \label{hydeJ_k} J_s\leq c_2(s)\zeta\norm{\Delta_h\Lambda^su}_{L^2_\omega(M)}+C_2(s,\z)\norm{u}_{L^2_\omega(M)}.
\end{equation}
Plugging (\ref{akhhh2}) and (\ref{hydeJ_k}) in (\ref{t5ntaktrro}, we get that
$$\exists c_1(s),c_2(s),\forall \z,\exists C_1(s,\z),C_2(s,\z),\forall h\leq 1,\forall u\in\mathcal{C}^\mathcal{1}(M),$$ 
\begin{equation}\label{minhohoho}
        \norm{\Delta_h\Lambda^su}_{L^2_\omega(M)}\leq (c_1(s)+c_2(s))\z\norm{\Delta_h\Lambda^su}_{L^2_\omega(M)}+(C_1(s,\z)+C_2(s,\z))\norm{u}_{L^2_\omega(M)}+\norm{\Delta_hu}_{H_{\omega}^s(M)}.
\end{equation}
Choose $\z$ small enough so that $(c_1(s)+c_2(s))\z<1$. Denote this $\z$ by $\z_0$. Then, (\ref{minhohoho}) implies that for any $h\leq1$, for any $u\in\mathcal{C}^\mathcal{1}(M)$, \begin{equation}
    \label{e5rwhdehonmf} \norm{\Delta_h\Lambda^su}_{L^2_\omega(M)}\leq \dfrac{C_1(s,\z_0)+C_2(s,\z_0)}{1-(c_1(s)+c_2(s))\z_0}\norm{u}_{L^2_\omega(M)}+\dfrac{1}{1-(c_1(s)+c_2(s))\z_0}\norm{\Delta_hu}_{H_{\omega}^s(M)}.
\end{equation}
Finally, plugging (\ref{e5rwhdehonmf}) into (\ref{a5}), we get that: $\exists c_3(s)>0\forall h\leq 1,\forall u\in\mathcal{C}^\mathcal{1}(M),$ \begin{equation*}
     \norm{u}_{H^{\epsilon+s}_\omega(M)}=\norm{\Lambda^su}_{H_{\omega}^\epsilon(M)}\leq c_3(s)\left(\norm{\Delta_hu}_{H_{\omega}^s(M)}+\norm{u}_{L^2_\omega(M)}\right).
\end{equation*}
This concludes the proof.
         \end{proof}
      \end{appendix}
\bibliographystyle{plain}
\bibliography{biblio_new.bib}
\end{document}